\documentclass{amsart}
\usepackage{latexsym,amsxtra,amscd,ifthen}
\usepackage{amsfonts}
\usepackage{verbatim}
\usepackage{amsmath}
\usepackage{amsthm}
\usepackage{amssymb}
\usepackage{url}
\usepackage{color}
\usepackage{tikz}

\usetikzlibrary{calc}

\theoremstyle{plain}

\numberwithin{equation}{section}

\newtheorem{theorem}[equation]{Theorem}
\newtheorem{lemma}[equation]{Lemma}
\newtheorem{proposition}[equation]{Proposition}
\newtheorem{corollary}[equation]{Corollary}

\theoremstyle{definition}
\newtheorem{definition}[equation]{Definition}
\newtheorem{example}[equation]{Example}
\newtheorem{remark}[equation]{Remark}
\newtheorem{question}[equation]{Question}
\newcommand{\assign}{:=}

\DeclareMathOperator{\LND}{LND}
\DeclareMathOperator{\End}{End}
\DeclareMathOperator{\tr}{tr}
\DeclareMathOperator{\Aut}{Aut}
\DeclareMathOperator{\gr}{gr}

\newcommand{\VA}{V_{n}({\mathcal A})}
\newcommand{\VQA}{V_{n}({\bf q},{\mathcal A})}

\newcommand{\betaAut}{\beta}

\newcommand{\rotsym}[1]{%
    \rotatebox{#1}{$\subseteq$}%
}
\newcommand{\rotsymp}[1]{%
    \rotatebox{#1}{$\supseteq$}%
}

\begin{document}

\title{Discriminant formulas and applications}

\author{K. Chan, A.A. Young, and J.J. Zhang}

\address{(Chan) Department of Mathematics, Box 354350,
University of Washington, Seattle, Washington 98195, USA}

\email{kenhchan@math.washington.edu, ken.h.chan@gmail.com}

\address{(Young) Department of Mathematics,
DigiPen Institute of Technology, Redmond, WA 98052, USA}

\email{young.mathematics@gmail.com}

\address{(Zhang) Department of Mathematics, Box 354350,
University of Washington, Seattle, Washington 98195, USA}

\email{zhang@math.washington.edu}

\begin{abstract}
We solve two conjectures of Ceken-Palmieri-Wang-Zhang 
concerning discriminants and give some applications.
\end{abstract}





\maketitle


\section*{Introduction}
\label{xxsec0}

In algebraic number theory, the discriminant takes on a familiar form: 
given a Galois extension $L$ of the field ${\mathbb Q}$ and write
${\mathcal O}_{L}={\mathbb Z}[\alpha]\cong {\mathbb Z}[x]/(f)$ where
$f$ is the minimal polynomial (or the characteristic polynomial) of 
$\alpha$, then  we have 
$$
\Delta_{L/{\mathbb Q}}=\prod_{i \neq j} ( r_{i} - r_{j} )$$ 
where $r_{1} , \ldots , r_{n}$ are the roots of $f$. In noncommutative 
algebra, the discriminant has long been used to study orders and lattices in a
central simple algebra \cite{Re}. Recently, it has been shown that the
discriminant plays a remarkable role in solving some classical 
and notoriously difficult questions:
\begin{enumerate}
\item[(1)]
{\bf Automorphism problem}, determining the full automorphism groups 
of noncommutative Artin-Schelter regular algebras \cite{CPWZ1, CPWZ2}. 
\item[(2)]
{\bf Zariski cancellation problem}, concerning the cancellative property
of noncommutative algebras such as skew polynomial rings \cite{BZ}.
\item[(3)]
{\bf Isomorphism problem}, finding a criterion for when two algebras are 
isomorphic, within certain classes of noncommutative algebras 
\cite{CPWZ3}.
\end{enumerate}
Despite the usefulness of the discriminant in algebraic number theory, 
algebraic geometry and noncommutative algebra, it is extremely hard 
to compute, especially in high dimensional and high rank cases. In 
\cite{CPWZ1, CPWZ2}, the authors made two conjectures on discriminant 
formulas for some classes of noncommutative algebras. Our main aim is 
to prove these two conjectures.

Let $k$ be a base commutative domain and $k^\times$ be the set of invertible
elements in $k$. The discriminant of a noncommutative algebra $A$ over
a central subalgebra $Z\subseteq A$, denoted by $d(A/Z)$, will be 
reviewed in Section \ref{xxsec1}. Let $q\in k^\times$ be an invertible
element in $k$ and let $A_q$ be the $q$-quantum Weyl algebra 
generated by $x$ and $y$ subject to the relation $yx=q xy+1$. Our first 
result is

\begin{theorem}
\label{xxthm0.1}
Let $q$ be a primitive $n$-th root of unity for some $n\geq 2$. Then the
discriminant of $A_q$ over its center $Z(A_q)$ is 
$$d(A_q/Z(A_q)) =c \; (nm)^{n^2}((1-q)^n x^n y^n-1)^{n(n-1)}$$
where $c$ is some element in $k^\times$ and $m =
\prod_{i=2}^{n-1} (1 + q + \cdots + q^{i-1})$. By convention, $m=1$ when $n=2$.
\end{theorem}

Theorem \ref{xxthm0.1} answers \cite[Conjecture 5.3]{CPWZ2} affirmatively. 

For $n\geq 2$, let $W_n$ be the $k$-algebra generated by
$x_1,\dots, x_n$ subject to the relations $x_i x_j+x_j x_i=1$ for all
$i\neq j$ \cite[Introduction]{CPWZ1}. This algebra is called a 
\emph{$(-1)$-quantum Weyl algebra} \cite[Introduction]{CPWZ3}. Let 
$$M:=\begin{pmatrix}
2x_1^2 & 1 & \cdots & 1\\
1 &2x_2^2& \cdots &  1\\
\vdots & \vdots &\cdots &\vdots\\
1& 1 & \cdots & 2x_n^2 \end{pmatrix}.$$
Let $Z$ denote the central subalgebra $k[x_1^2, \cdots, x_n^2] \subseteq 
W_n$. Our second result is

\begin{theorem}
\label{xxthm0.2}
Suppose $2$ is invertible in $k$. Then the 
discriminant of $W_n$ over the subalgebra $Z$ is 
$$d(W_n/Z)=c \; (\det M)^{2^{n-1}}$$
where $c$ is an element in $k^\times$.
\end{theorem}

Theorem \ref{xxthm0.2} answers \cite[Question 4.12(2)]{CPWZ1} affirmatively.

These results suggest that the discriminant has elegant
expressions in some situations. Because of its usefulness, more 
discriminant formulas should be established, see Example \ref{xxlem6.4}.

This paper contains other related results which we now describe.
Let $T$ be a commutative algebra over $k$ and let ${\bf q}:=\{q_{ij}\in T^\times 
\mid 1\leq i<j \leq n\}$ and ${\mathcal A}:=\{a_{ij}\in T\mid 1\leq i<j 
\leq n\}$ be sets of elements in $T$. The skew polynomial ring 
$T_{{\bf q}}[x_1,\cdots,x_n]$ is a $T$-algebra generated by $x_1,\cdots,x_n$ 
subject to the relations
\begin{equation}
\label{E0.2.1}\tag{E0.2.1}
x_j x_i=q_{ij}x_i x_j, \quad \forall \; 1\leq i<j\leq n.
\end{equation}
A generalized quantum Weyl algebra associated to $({\bf q},{\mathcal A})$ is a 
$T$-central filtered algebra of the form
\begin{equation}
\label{E0.2.2}\tag{E0.2.2}
\VQA \; = \; 
\frac{T \langle x_1, \ldots ,x_n\rangle}{(x_j x_i -q_{ij} x_i x_j -a_{i j} 
\mid i < j )}  
\end{equation}
such that the associated graded ring $\gr \VQA$ is naturally isomorphic 
to the skew polynomial ring $T_{\bf q}[x_1,\cdots,x_n]$. Another way of
constructing $\VQA$ is to use an iterated Ore extension starting with $T$.
To calculate the discriminant of $\VQA$ over its center, one needs to
determine the center of $\VQA$. The discriminant is defined whenever $\VQA$ 
is a finite module over a central subring $Z$ \cite{CPWZ2}, and it is 
most useful when $\VQA$ is a free module over $Z$ \cite{CPWZ1}. Since
$\gr \VQA \simeq T_{\bf q} [x_1, \cdots, x_n]$ we have that $\gr \VQA$
is a finite module over its center if and only if each $q_{i j}$ is a 
root of unity. Using this, we can show that the algebra $\VQA$ is a 
finite module over its center if and only if the parameters $q_{i j}$ 
are all non-trivial roots of unity. Also, when the center 
of $\VQA$ is a polynomial ring, $\VQA$ is a finitely generated free 
module over its center. The following useful result concerns the centers
of $\VQA$ and $T_{\bf q}[x_1,\cdots,x_n]$.

To state it, we need some notation. When $q_{ij}$ is a root of unity, 
there are two integers $k_{ij}$ and $d_{ij}$ such that 
$$q_{i j} = \exp ( 2 \pi\sqrt{-1}\; k_{i j} /d_{i j} ),$$
where $d_{i j}: =o ( q_{i j} ) <\infty$, $|k_{ij}|<d_{ij}$ and 
$( k_{ij} ,d_{i j} ) =1$. Further, we can choose that $k_{i j} =-k_{j i}$ 
since $q_{ji}=q_{ij}^{-1}$. Let $L_i$ be the $\mathrm{lcm} \{d_{i j} 
\mid j=1, \ldots ,n \}$. Let $\overline{Y}$ be the $n\times n$-matrix
$(k_{ij}L_i/d_{ij})_{n\times n}.$
For each prime $p$, define $\overline{Y}_p = \overline{Y} \otimes \mathbb{F}_p$. 
Let $m$ be any natural number. Let $I_{p,m}$ be the set containing $i$ 
such that $L_i\in p^{m}{\mathbb Z}-p^{m+1}{\mathbb Z}$. Finally
let $\overline{Y}_{p,m}$ be the submatrix of $\overline{Y}_{p}$ taken 
from the row and columns with indices $i\in I_{p,m}$.

\begin{theorem}
\label{xxthm0.3} 
Suppose $q_{ij}$ is a root of unity and not 1 for all $i<j$.
\begin{enumerate}
\item[(1)]
The center of $T_{\bf q}[x_1,\cdots,x_n]$ is a polynomial ring if and only if 
it is of the form $T[x_1^{L_1},\cdots, x_n^{L_n}]$ if and only if 
$\det(\overline{Y}_{p,m})\neq 0$ in ${\mathbb F}_p$ for all primes $p$ and all 
integers $m>0$ such that $I_{p,m} \ne \emptyset$.
\item[(2)]
If the center of $T_{\bf q}[x_1,\cdots,x_n]$ is the subalgebra 
$T[x_1^{L_1},\cdots,x_n^{L_n}]$, then the center of $\VQA$ is the 
subalgebra $T[x_1^{L_1},\cdots,x_n^{L_n}]$ and $\VQA$ is finitely
generated and free over $T[x_1^{L_1},\cdots,x_n^{L_n}]$.
\end{enumerate}
\end{theorem}

The above criterion can be simplified when $n=3$ or $4$ 
[Corollaries \ref{xxcor5.4} and \ref{xxcor5.5}]. The point of
Theorem \ref{xxthm0.3} is that it provides an explicit linear 
algebra criterion for when the center of $T_{\bf q}[x_1,\cdots,x_n]$
is isomorphic to a polynomial ring. One interesting question 
after this is the following.

\begin{question}
\label{xxque0.4} Suppose that $A:=\VQA$ is finitely
generated and free over its center $Z$. What is the
discriminant $d(A/Z)$?
\end{question}

Theorems \ref{xxthm0.1} and \ref{xxthm0.2} answer this 
question for two special cases.

A secondary goal of this paper is to provide some quick applications.
These discriminant formulas have potential applications in algebraic 
geometry, number theory and the study of Clifford algebras. 
In Section \ref{xxsec8} (the final section), we give some immediate 
applications of discriminants to the cancellation problem and 
the automorphism problem for several classes of noncommutative 
algebras. 

Let us briefly review some definitions. An algebra $A$ is called 
{\it cancellative} if $A[t]\cong B[t]$ for some algebra $B$ implies
$A\cong B$. Let $\Aut(A)$ be the group of all algebra automorphisms
of $A$. Let $A$ be connected graded. An algebra automorphism $g$ of 
$A$ is called {\it unipotent} if 
$$g(v) = v + {\text{(higher degree terms)}}$$
for all homogeneous elements $v\in A$. Let $\Aut_{uni}(A)$ denote
the subgroup of $\Aut(A)$ consisting of all unipotent automorphisms
\cite[After Theorem 3.1]{CPWZ2}. When $\Aut_{uni}(A)$ is trivial,
$\Aut(A)$ is usually small and easy to handle. We will give a criterion
on when $\Aut_{uni}(A)$ is trivial. 

Let $A$ be a domain and $F$ be a 
subset of $A$. Let $Sw(F)$ be the set of $g\in A$ such that $f=agb$ 
for some $a,b\in A$ and $0\neq f\in F$. Let $D_1(F)$ be the 
$k$-subalgebra of $A$ generated by $Sw(F)$. For $n>2$, we define 
$D_n(F)=D_1(D_{n-1}(F))$ inductively, and define
$D(F)=\bigcup_{n\geq 1} D_n(F)$. This algebra is called the 
the {\it $F$-divisor subalgebra of $A$}. When $F=\{d(A/Z)\}$,
$D(F)$ is called {\it discriminant-divisor subalgebra of $A$}
and is denoted by ${\mathbb D}(A)$.
The main result in Section \ref{xxsec8} is the following.

\begin{theorem}
\label{xxthm0.5} Suppose $k$ is a field of characteristic zero.
Let $A$ be a connected graded domain of finite
Gelfand-Kirillov dimension. Assume that $A$ is finitely 
generated and free over its center. If ${\mathbb D}(A)=A$, then $A$ is 
cancellative and $\Aut_{uni}(A)=\{1\}$.
\end{theorem}

The above theorem can be applied to some Artin-Schelter regular algebras
of global dimension four in Examples \ref{xxex6.3} and \ref{xxex8.4}.
Further applications are certainly expected.

This paper is organized as follows. Background material about discriminants
is provided in Section \ref{xxsec1}. We prove Theorem \ref{xxthm0.1} in
Section \ref{xxsec2} and Theorem \ref{xxthm0.2} in Section \ref{xxsec3}.
Sections \ref{xxsec4}-\ref{xxsec6} concern the question of when 
$T_{\bf q}[x_1,\cdots,x_n]$ and $\VQA$ are finitely generated and free over
their centers and contain the proof of Theorem \ref{xxthm0.3}. In
Section \ref{xxsec7}, we review and introduce some invariants 
related to discriminants, locally nilpotent 
derivations, and automorphisms, which will be used in Section 
\ref{xxsec8}. In Section \ref{xxsec8}, some applications are 
provided and Theorem \ref{xxthm0.5} is proven.

\section{Preliminaries}
\label{xxsec1}
In this section we recall some definitions and basic properties of 
the discriminant. A basic reference is \cite[Section 1]{CPWZ1}.

Throughout let $k$ be a base commutative domain and everything is over
$k$. Let $A$ be an algebra and $Z$ be a central subalgebra of $A$
such that $A$ is finitely generated and free over $Z$. A modified version 
of the discriminant was introduced in \cite{CPWZ2} when $A$ is not free
over $Z$; however, in this paper, we only consider the case when $A$ is 
finitely generated and free over $Z$. Let $r$ be the rank of $A$ over $Z$. 

We embed $A$ in the endomorphism ring $\End(A_Z)$ by sending 
$a\in A$ to the left multiplication $l_a: A\to A$. Since $A$ is a free
over $Z$ of rank $r$, $\End(A_Z)\cong M_{r\times r}(Z)$. Define the
trace function
\begin{equation}
\label{E1.0.1}\tag{E1.0.1}
\tr: A\to \End(A_Z)\cong M_{r\times r}(Z) \xrightarrow{tr_m} Z
\end{equation}
where $tr_{m}$ is the usual matrix trace. The trace function $\tr$
is independent of the choice of basis of $A$ over $Z$. 

\begin{definition} \cite[Definition 1.3(3)]{CPWZ1}
\label{xxdef1.1} Retain the above notation. Suppose that $A$ is a free module
over a central subalgebra $Z$ with a $Z$-basis $\{z_1,\cdots, z_r\}$.
The discriminant of $A$ over $Z$ is defined to be
$$d(A/Z)=\det (\tr(z_i z_j))_{r\times r} \in Z.$$
\end{definition}

By \cite[Proposition 1.4(2)]{CPWZ1}, $d(A/Z)$ is unique up to a scalar
in $Z^\times$. For $x,y\in Z$, we use the notation
$x=_{Z^\times} y$ to indicate that $x=cy$ for some $c\in Z^{\times}$.
So $d(A/Z)=_{Z^{\times}} \det (\tr(z_i z_j))_{r\times r} $ as in
\cite[Definition 1.3(3)]{CPWZ1}. 
The following lemma is easy.

\begin{lemma}
\label{xxlem1.2} Retain the notation as in Definition \ref{xxdef1.1}.
Let $(A',Z')$ be another pair of algebras such that $Z'$ is a central
subalgebra of $A'$ and $A'$ is a free $Z'$-module of rank $r$.
Let $g: A\to A'$ be an algebra homomorphism such that 
\begin{enumerate}
\item[(i)]
$g(Z)\subseteq Z'$.
\item[(ii)]
$\{g(z_1), \cdots, g(z_r)\}$ is a $Z'$-basis of $A'$.
\end{enumerate}
Then $g(d(A/Z))=_{(Z')^{\times}} d(A'/Z')$.
\end{lemma}

\begin{proof} For any $a\in A$, we denote $a'=g(a)$. 
Write $a z_i= \sum_{j=1}^r a_{ij} z_j$ for all $i$. By applying
$g$ to the last equation, we have $a' z'_i= \sum_{j=1}^r a'_{ij} 
z'_j$. By definition \eqref{E1.0.1}, $\tr(a)=\sum_{i} a_{ii}$ and
$$\tr(g(a))=\tr(a')=\sum_i a'_{ii}=g(\sum_i a_{ii})=g(\tr(a))$$
for all $a\in A$. By Definition \ref{xxdef1.1} and the above equation,
$$g(d(A/Z))=g(\det(\tr(z_iz_j))_{r\times r})
=\det(\tr(z'_iz'_j))_{r\times r}=_{(Z')^{\times}} d(A'/Z').$$
\end{proof}

Let $Z$ be a central subalgebra of $A$ and consider an Ore set 
$C\subset Z$. Then the localization $ZC^{-1}$ is central in 
$AC^{-1}$. 

\begin{lemma}
\label{xxlem1.3} 
Let $Z$ be a central subalgebra of $A$. Suppose $A$ is free over 
$Z$ of rank $r$. Then $AC^{-1}$ is free over $ZC^{-1}$ of rank $r$. 
As a consequence,
$$d(AC^{-1}/ZC^{-1})=_{(ZC^{-1})^{\times}} d(A/Z).$$
\end{lemma}

\begin{proof} Let $\{z_1, \cdots, z_r\}$ is a $Z$-basis of 
$A$. Then it is also a $ZC^{-1}$-basis of $AC^{-1}$. The 
consequence follows from Lemma 1.2. 
\end{proof}

We will need the following result from \cite[Proposition 2.8]{CPWZ2}.
We change notation from $k$ to $T$ to denote a commutative
domain in the following proposition.

\begin{proposition} \label{xxpro1.4}
Let $T$ be a commutative domain and let
$A=T_{\bf q}[x_1,\cdots,x_n]$. Suppose $Z:=T[x_1^{\alpha_1},\cdots,
x_n^{\alpha_n}]$ is a central subalgebra of $A$, where the
$\alpha_i$ are positive integers. 
\begin{enumerate}
\item[(1)]\cite[Proposition 2.8]{CPWZ2}
Let $r=\prod_{i=1}^n \alpha_i$. Then 
$$d(A/Z)=_{T^{\times}} r^r (\prod_{i=1}^n x_i^{\alpha_i-1})^r.$$
\item[(2)]
If $n=2$ and $q_{12}$ is a primitive $m$-th root of unity and
$Z=T[x_1^m, x_2^m]$, then
$$d(A/Z)=_{T^{\times}} m^{2m^2}(x_1^m x_2^m)^{m(m-1)}.$$
\item[(3)]
If $q_{ij}=-1$ for all $i<j$ and $\alpha_i=2$ for all $i$, then
$$d(A/Z)=_{T^{\times}} 2^{n 2^n}(\prod_{i=1}^n x_i^2)^{2^{n-1}}.$$
\end{enumerate}
\end{proposition}

\begin{proof} Parts (2,3) are special cases of part (1).
\end{proof}

The next lemma is a special case \cite[Proposition 4.10]{CPWZ2}. Suppose
$Z$ is a central subalgebra of $A$ and $A$ is free over $Z$ of rank
$r<\infty$. We fix a $Z$-basis of $A$, say $b:=\{b_1=1, b_2, 
\cdots, b_r\}$. Suppose $A$ is an ${\mathbb N}$-filtered algebra such
that the associated graded ring $\gr A$ is a domain. For any element
$f\in A$, let $\gr f$ denote the associated element in $\gr A$. Let $\gr b$
denote the set $\{\gr b_1, \cdots , \gr b_r\}$, which is a subset
of $\gr A$. 

\begin{lemma}
\cite[Proposition 4.10]{CPWZ2}
\label{xxlem1.5} Retain the above notation. Suppose that
$\gr A$ is finitely generated and free over $\gr Z$ with basis
$\gr b$. Then 
$$\gr\; (d(A/Z))=_{(\gr Z)^{\times}} d(\gr A/\gr Z).$$
\end{lemma}

\section{Discriminant of $A_q$ over its center}
\label{xxsec2}

Let $T$ be a commutative domain and $q \in T^{\times}$ be a primitive 
$n$-th root of unity for some $n\geq 2$. Let $A_{q}$ be the $q$-quantum 
Weyl algebra over $T$ generated by $x$ and $y$ subject to the relation 
$yx=qxy+a$ for some $a\in T$. This agrees with the definition of $A_q$ 
given in the introduction when $T=k$ and $a=1$. It is easy to check that 
the center of $A_q$, denoted by $Z(A_q)$, is $T[x^n, y^n]$, and that 
$A_q$ is free over $Z(A_q)$ of rank $n^2$. A $Z(A_q)$-basis of 
$A_q$ is $\mathcal{B} := \{x^i y^j \mid 0\leq i,j \leq n-1\}$. The aim of 
this section is to compute the discriminant $d(A_q/Z(A_q))$.

Let $A'$ be the $T$-subalgebra of $A_q$ generated by $x':=(1-q)x$ and $y$.
Since $yx'=q x'y+ (1-q)a$ and $(1-q)$ may not be  invertible, there is no 
obvious algebra homomorphism from $A_q$ to $A'$. Let
$Z'$ be the subalgebra $T[(x')^n,y^n]$ which is the center of $A'$. 

\begin{lemma}
\label{xxlem2.1}
Retain the above notation. Then
\begin{eqnarray*}
    d ( A' /Z' ) & = & (1-q)^{n^2(n-1)} d (A_q/Z(A_q)).
\end{eqnarray*}
\end{lemma}

\begin{proof}
Let $\tr': A'\to Z'$ be the trace function defined as in \eqref{E1.0.1}.
We use this trace function to compute the discriminant $d(A'/Z')$.

Let $\mathcal{B}' := \{ (x')^i y^j\}_{0\leq i,j\leq n-1}$. Then 
$\mathcal{B}'$ is a $Z'$-basis of $A'$. 
Note that $A'$ and $A_q$
have the same ring of fractions and $Z(A_q)$ and $Z'$ have the
same fraction field. Since the trace function is independent of
the choice of basis we have $\tr'(a)=\tr(a)$ for all
$a\in A'$.

Picking any two elements $b_s=x^{i_s} y^{j_s}$ and $b_t=x^{i_t} y^{j_t}$
in $\mathcal{B}$, we have corresponding elements 
$b'_s=(x')^{i_s} y^{j_s}$ and $b'_t=(x')^{i_t} y^{j_t}$ in $\mathcal{B}$. 
Hence
$$\tr'(b'_s b'_t)=\tr((1-q)^{i_s+i_t} b_sb_t)
=(1-q)^{i_s+i_t}\tr(b_sb_t).$$
By definition, $d(A'/Z')=\det [\tr'(b'_s b'_t)_{b'_s,b'_t\in \mathcal{B}'}]$.
Hence we have
$$\begin{aligned}
d(A'/Z')&=\det [(\tr'(b'_s b'_t))_{b'_s,b'_t\in \mathcal{B}'}]\\
&= \det [((1-q)^{i_s+i_t}\tr(b_s b_t))_{b_s,b_t\in \mathcal{B}} ]\\
&= (1-q)^N \det [(\tr(b_s b_t))_{b_s,b_t\in \mathcal{B}} ]\\
&= (1-q)^N d(A_q/Z(A_q)),
\end{aligned}
$$
where
$$N=\sum_{{\text{all}}\; i_s, i_t} (i_s+i_t)=2\sum_{{\text{all}}\; i_s} i_s
=2 n (0+1+2+\cdots+(n-1))=n^2 (n-1).$$
The assertion follows.
\end{proof}

Following the above lemma, we first compute $d(A'/Z')$. We can re-write
$A'$ as $T\langle x', y\rangle/(yx'-q x'y -(1-q)a)$ so that the positions
of $x'$ and $y$ are more symmetrical. 

Let $C=\{(y^n)^i\mid i\geq 1\}$. Consider the localizations 
$Z'':=Z'C^{-1}$ and $A'':=A' C^{-1}$. Let 
$$x'' := x' - a y^{-1}=(1-q)x - (a y^{-n}) y^{n-1} \in A''.$$

\begin{lemma}
\label{xxlem2.2}
Retain the above notation. The following hold:
\begin{enumerate}
\item[(1)] 
$y x''-q x'' y=0$. 
\item[(2)]
$A'':=A'C^{-1}$ is generated by $T$, $(y^{n})^{-1}$, $x''$ and $y$.
\item[(3)] 
$(x'')^{n}$ is central and $d(A''/Z'')=_{(Z'')^{\times}} 
n^{2n^2}((x'')^n y^n)^{n(n-1)}$.
\item[(4)] 
$ d(A''/Z'')=_{(Z'')^{\times}} 
n^{2n^2}((1-q)^n x^n y^n - a^{n} )^{n(n-1)}.$
\end{enumerate}
\end{lemma}

\begin{proof}
(1) We have $yx''-q x''y=y((1-q)x - a y^{-1})-q((1-q)x - a y^{-1})y = 0$.

(2) This is clear.

(3) Since $q^{n} =1$, $(x'')^n$ commutes with $y$ by part (1).
By part (2), $(x'')^n$ commutes with every element in $A''$.

Consider an algebra homomorphism $g: T_q[x_1,x_2]\to A''$ determined by
$g(x_1)=x''$ and $g(x_2)=y$. Then the center of $B:=T_q[x_1,x_2]$
is $R:=T[x_1^n,x_2^n]$ and $\{x_1^i x_2^j\mid 0\leq i,j \leq n-1\}$
is an $R$-basis of $B$. It is clear that $A''$ is free of rank 
$n^2$ and $A''=\sum_{0\leq i,j \leq n-1} (x')^i y^j Z''$. Hence 
$\{(x'')^i y^j\mid 0\leq i,j \leq n-1\}$ is a $Z''$-basis of $A''$. 
Then the hypotheses of Lemma \ref{xxlem1.2} hold. Applying 
Lemma \ref{xxlem1.2} to $g$, we have $g (d(B/R))=_{(Z')^{\times}} d(A''/Z'')$.
By Proposition \ref{xxpro1.4}(2), $d(B/R)=n^{2n^2} (x_1^n x_2^n)^{n(n-1)}$.
Therefore, $d(A''/Z'')=_{(Z')^{\times}} n^{2n^2} ((x'')^n y^n)^{n(n-1)}$. 

(4) In the following, we will denote $\psi =y^{-1}$, $z=x''$ and $p=q^{-1}$. 
The commutation relation between $x'$ and $\psi$ is
\begin{equation}
\label{E2.2.1}\tag{E2.2.1}
\psi x' = (1-q)\psi x  = (1-q)(p x \psi -p a \psi^{2}) = p x' \psi - (p-1) a \psi^2 .
\end{equation}

Recall that $z=x''=x' - a\psi$. Write $z^{n}=\sum_{i=0}^{n} c_{i} (x')^{i} 
\psi^{n-i}$. Since $z^{n}$ is central (see part (3)), we have $c_{i} =0$ 
unless $i=0,n$. It is clear that $c_{n} =1$. Next we determine $c_0$.
Since $A''$ is a free module over $Z''$ with basis $\{(x')^i \psi^j 
\mid 0\leq i,j\leq n-1\}$, we can work modulo the right $Z''$-submodule 
$W$ generated by $(x')^{i} \psi^{j}$ where $0<i<n$ and $0 \leq j<n$.
Let $\equiv$ denote equivalence mod $W$. 

By induction, for $i=1, \ldots ,n-1$, we have
\begin{equation}
\label{E2.2.2}\tag{E2.2.2}
\psi^{i} x' = p^i x' \psi^{i} -(p^{i} - 1)(a \psi^{i+1}). 
\end{equation}

Then $\psi^i x' \equiv -(p^{i} - 1)(a \psi^{i+1})$. 
For each $1\leq j\leq n-1$, write 
$$z^j=\sum_{i=0}^j c_i^j (x')^i \psi^{j-i}.$$ 
Then $x' z^j\in W$ 
for all $j<n-1$ and $x' z^{n-1}\equiv (x')^n$. For each $j$, we have
$\psi^{j-1} z^{n-j}=\sum_{i=0}^{n-j} d^j_i (x')^i \psi^{n-1-i}$ for some
$d^j_i\in Z'$, so
\begin{equation}
\label{E2.2.3}\tag{E2.2.3}
x'\psi^{j-1} z^{n-j}\in W
\end{equation}
for all $j\geq 2$.  By the
above computation and \eqref{E2.2.1}-\eqref{E2.2.3}, we have
$$\begin{aligned}
z^{n} - (x')^n & = (x' - a \psi ) z^{n-1} - (x')^n \\
&=x' z^{n-1} - (x')^n - a \psi z^{n-1}\\
& \equiv - a \psi (x' - a \psi) z^{n-2}\\
& \equiv - a (p x' \psi - (p-1) a \psi^2 - a \psi^2) z^{n-2}\\
& \equiv -a ( -p a ) \psi^{2} z^{n-2} - ap x' \psi z^{n-2}\\
& \equiv -a ( -p a ) \psi^{2} z^{n-2}\\
& \equiv -a ( - pa ) ( \psi^{2} x - a \psi^{3} ) z^{n-3}\\
&\equiv  -a ( - pa ) ( -p^{2} a ) \psi^{3} z^{n-3}\\
      & \qquad \vdots \\
&\equiv -a ( - pa ) ( -p^{2} a ) \cdots ( -p^{n-1} a )
\psi^{n} \\ 
&= (-a)^n p^{(n-1)n/2} \psi^{n} = - a^n \psi^{n}.
\end{aligned}
$$

Therefore
$$z^{n} \equiv - a^{n} \psi^{n} + (x')^n.$$
Hence $c_{0} = - a^{n}$ and $z^n = (x')^n - a^n \psi^n$. Combining all the above, we have
$$(x'')^n y^n=((x')^n - a^n \psi^n) y^n = (x')^n y^n - a^n
= (1-q)^n x^n y^n-a^n.$$
Part (4) follows from part (3) and the above formula.  
\end{proof}

\begin{lemma}
\label{xxlem2.3}
The discriminant of $A'$ over its center $Z'$ is 
$$d(A'/Z') =_{T^{\times}} \; n^{2n^2}((1-q)^n x^n y^n-a^n)^{n(n-1)}.$$
\end{lemma}
\begin{proof} 
Let $g$ be the embedding of $A'$ into $A''=A' C^{-1}$,
viewed as an inclusion. By Lemma \ref{xxlem1.2}, $g$ sends 
$d(A'/Z')$ to $d(A''/Z'')$. Combining with Lemma \ref{xxlem2.2}(4),
we have 
$$\begin{aligned}
d(A'/Z')&=_{(Z'')^{\times}}
g(d(A'/Z(A')))=_{(Z'')^{\times}} d(A''/Z'')\\
&=_{(Z'')^{\times}} n^{2n^2}((1-q )^{n} x^n y^n -a^{n} )^{n(n-1)}.
\end{aligned}
$$
Let $\Phi$ be the element $d(A'/Z') 
\{n^{2n^2}((1-q)^n x^n y^n-a^n)^{n(n-1)}\}^{-1}$, which can be viewed as an 
element in the quotient ring of $A'$. By the above equation,
$\Phi$ is in $(Z'')^{\times}$. Since $Z''=T[(x')^n, y^{\pm n}]$, $\Phi$ is 
of the form $\alpha y^{sn}$ for some $\alpha\in T^{\times}$ and 
some $s$. By symmetry, $\Phi$ is also of the form $\beta 
(x')^{tn}$ for some $\beta\in T^{\times}$ and some $t$. Hence
$s=t=0$, $\alpha=\beta\in T^{\times}$ and $\Phi=\alpha\in T^{\times}$. 
Therefore $d(A'/Z') = \alpha n^{2n^2}((1-q)^n x^n y^n-a^n)^{n(n-1)}$
and the assertion follows. 
\end{proof}

Now let
\begin{equation}
\label{E2.3.1}\tag{E2.3.1}
m := 
\prod_{i=2}^{n-1} (1 + q + \cdots + q^{i-1}).
\end{equation}
We can show that $n = (1 - q)^{n-1} m$ by factoring the polynomial 
$x^n - 1 \in T[x]$, dividing by $(x-1)$, then substituting $1$ for 
$x$ as follows:
$$x^n - 1 = \prod_{i=0}^{n-1} (x - q^i),$$
$$\sum_{i=0}^{n-1} x^i = \frac{x^n-1}{x-1} = \prod_{i=1}^{n-1} (x - q^i),$$
\begin{equation}
\label{E2.3.2}\tag{E2.3.2}
n = \prod_{i=1}^{n-1} (1 - q^i) = (1-q)^{n-1} \prod_{i=2}^{n-1} 
(1 + q + \cdots + q^{i-1}) = (1-q)^{n-1} m.
\end{equation}

Now we are ready to prove the main result of this
section, that also recovers Theorem \ref{xxthm0.1}.

\begin{theorem}
\label{xxthm2.4} Retain the above notation.
The discriminant of $A_q$ over its center $Z(A_q)$ is
$$d(A_q/Z(A_q)) =_{T^\times} (nm)^{n^2} ((1-q)^n x^n y^n-a^n)^{n(n-1)}.$$
\end{theorem}

\begin{proof}
Using Lemmas \ref{xxlem2.1} and \ref{xxlem2.3} and equation \eqref{E2.3.2}, 
we have
$$(1-q)^{n^2(n-1)} d(A_q/Z(A_q)) 
=_{T^\times} (nm(1-q)^{n-1})^{n^2} ((1-q)^n x^n y^n-a^n)^{n(n-1)}.$$
Since $A_q$ is a domain, we obtain that
$$d(A_q/Z(A_q)) =_{T^\times} (nm)^{n^2} ((1-q)^n x^n y^n-a^n)^{n(n-1)}.$$
\end{proof}

\begin{remark}
\label{xxrem2.5}\quad
\begin{enumerate}
\item[(1)]
By \cite[Lemma 2.7(7)]{CPWZ2}, the integer $n$ in Theorem \ref{xxthm2.4}
is nonzero in $T$. However $n$ and $m$ may not be invertible in general.
\item[(2)]
Theorem \ref{xxthm0.1} is clearly a consequence of
Theorem \ref{xxthm2.4}.
\end{enumerate}
\end{remark}

A slight generalization of Theorem \ref{xxthm2.4} is the following.

\begin{theorem}
\label{xxthm2.6}
Let $T$ be a commutative domain and $q\in T^{\times}$ be a primitive
$n$-th root of unity. Let 
$B$ be the $T$-algebra of the form
$$\frac{T\langle x, y\rangle}{(yx-qxy=a, x^n=b, y^n=c)}$$
where $a,b,c\in T$. Suppose that $B$ is a free module over $T$ with 
basis $\{x^i y^j \mid 0\leq i,j \leq n-1\}$. Then $d(B/T)=_{T^{\times}}
(nm)^{n^2}((1-q)^n x^n y^n-a^n)^{n(n-1)}$, where $m$ is given in \eqref{E2.3.1}.
\end{theorem}

\begin{proof} First note it is well-known and easy to check that $T$ 
is the center of $B$. 

Recall that $A_q$ is the algebra of the form 
$T\langle x, y\rangle/(yx-qxy=a)$. There is a 
natural algebra homomorphism $g$ from $A_q$ to $B$ sending $x$
to $x$ and $y$ to $y$ and $t\in T$ to $t\in T$. Then the 
hypotheses in Lemma \ref{xxlem1.2} hold. By Lemma 
\ref{xxlem1.2}, $g(d(A_q/Z(A_q)))= d(B/T)$. Now the assertion
follows from Theorem \ref{xxthm2.4}.
\end{proof}

\section{Discriminant of Clifford algebras}
\label{xxsec3}

In this section we assume that $2^{-1}\in k$. We fix an integer 
$n\geq 2$. 

Let $T$ be a commutative domain and let ${\mathcal A}:=
\{ a_{ij}\mid 1\leq i<j\leq n\}$ be a set of scalars in $T$. 
We write $a_{ji}=a_{ij}$ if $i<j$. Let 
$\VA$ be the $T$-algebra generated by $\{x_1,\cdots,x_n\}$ 
subject to the relations
$$x_ix_j+x_jx_i=a_{ij}, \; \forall\;  i\neq j.$$
This algebra was studied in \cite{CPWZ1, CPWZ3}. Some 
basic properties of $\VA$ are given in \cite[Section 4]{CPWZ1}.
Let $M_1$ be the matrix
\begin{equation}
\label{E3.0.1}\tag{E3.0.1}
M_1:=\begin{pmatrix}
2x_1^2 & a_{12} & \cdots & a_{1n}\\
a_{21} &2x_2^2& \cdots &  a_{2n}\\
\vdots & \vdots &\cdots &\vdots\\
a_{n1}& a_{n2} & \cdots & 2x_n^2 \end{pmatrix}.
\end{equation}
This is a symmetric matrix
with entries in $Z:=T[x_1^2,\cdots,x_n^2]$. We will define a sequence of 
matrices $M_i$ later. Note that $Z$ is a central
subalgebra of $\VA$. If we write $M_1=(m_{ij,1})_{n\times n}$, then
$m_{ij,1}=x_jx_i+x_ix_j$ for all $i,j$. 

The algebra $V_n({\mathcal A})$ is a Clifford algebra over $Z$.
We will recall the definition of the Clifford algebra associated to a 
quadratic form in the second half of this section. In the next few lemmas, we 
are basically diagonalizing the quadratic form, which is elementary 
and well-known in the classical case, see 
\cite[Chapter I, Corollary 2.4]{La} for some related material.
Since we need some explicit construction to complete the 
proof of our main result, details will be provided below.

We will introduce a sequence of new variables starting with
$$x_{i,1}=x_i, \; \forall \; i=1,\cdots,n,$$
and
$$a_{ij,1}=a_{ij}, \; \forall\; i\neq j, \quad {\text{and}}\quad
a_{ii,1}= 2x_i^2, \; \forall \; i.$$
So we have $x_{j,1}x_{i,1}+x_{i,1}x_{j,1}=a_{ij,1}$ for
all $i,j$. Let 
\begin{equation}
\label{E3.0.2}\tag{E3.0.2}
{\text{$x_{1,2}:=x_{1,1}$ and 
$x_{i,2}:=x_{i,1}-\frac{1}{2} a_{1i,1} x_{1,1}^{-2} x_{1,1}$ 
for all $i\geq 2$.}}
\end{equation}

\begin{lemma}
\label{xxlem3.1} Retain the notation as above.
\begin{enumerate}
\item[(1)]
$x_{i,2} x_{1,2}+x_{1,2} x_{i,2}=0$ for all $i\geq 2$.
\item[(2)]
$x_{i,2}^2=x_{i,1}^2-\frac{1}{4}a_{1i,1}^2 x_{1,1}^{-2}$ for all $i\geq 2$.
\item[(3)]
$x_{i,2}x_{j,2}+x_{j,2}x_{i,2}=a_{ij,1}-\frac{1}{2} 
a_{1i,1}a_{1j,1} x_{1,1}^{-2}$ for all $2\leq i<j\leq n$.
\item[(4)]
Let $M_2$ be the matrix $( x_{i,2}x_{j,2}+x_{j,2}x_{i,2})_{1\leq i,j\leq n}$.
Then $\det M_2=\det M_1$.
\item[(5)]
Let $C_1=\{x_{1,1}^{2i}\}_{i\geq 1}$. Then the localization $\VA[C_1^{-1}]$
is free over $Z[C_1^{-1}]$ with basis $\{x_{1,2}^{d_1}\cdots x_{n,2}^{d_n}
\mid d_s=0,1\}$. 
\end{enumerate}
\end{lemma}

\begin{proof} (1,2,3) Follows by direct computation.

(4) Let $N$ be the matrix 
$$\begin{pmatrix} 1 & 0 &0 &\cdots &0\\
-\frac{1}{2} a_{12,1} x_{1,1}^{-2} & 1 & 0& \cdots &0\\
-\frac{1}{2} a_{13,1} x_{1,1}^{-2} & 0 & 1& \cdots &0\\
\vdots & \vdots &\vdots &\cdots & \vdots\\
-\frac{1}{2} a_{1n,1} x_{1,1}^{-2} & 0 & 0& \cdots &1
\end{pmatrix}.
$$ By linear algebra and part (3), one can check that
$N M_1 N^T=M_2$. Since $\det N=1$, we have $\det M_2=\det M_1$.

(5) First of all $\VA$ is free over $Z$ with basis 
$\{x_{1,1}^{d_1}\cdots x_{n,1}^{d_n} \mid d_s=0,1\}$. In the localization
$\VA[C^{-1}]$, this basis can be transformed to a basis 
$\{x_{1,2}^{d_1}\cdots x_{n,2}^{d_n} \mid d_s=0,1\}$ by using
\eqref{E3.0.2}.
\end{proof}

After we have $x_{i,2}$, define $a_{ij,2}$ to be $x_{i,2}x_{j,2}
+x_{j,2}x_{i,2}$ for all $i,j$. 
Now we define $x_{i,s}$ and $a_{ij,s}$ inductively. 

\begin{definition}
\label{xxdef3.2} Let $s\geq 3$ and suppose that
$x_{i,s-1}$ and $a_{ij,s-1}$ are defined inductively. 
Define 
\begin{equation}
\label{E3.2.1}\tag{E3.2.1}
x_{i,s}:=x_{i,s-1}, \; \forall\; i<s \quad
{\text{and}}\quad 
x_{i,s}:=x_{i,s-1}-\frac{1}{2} a_{s-1 i, s-1} x_{s-1,s-1}^{-1}, \;
\forall i\geq s.
\end{equation}
Define $a_{ij,s}:=x_{i,s}x_{j,s}+x_{j,s}x_{i,s}$ for all $i,j$.
\end{definition}

Similar to Lemma \ref{xxlem3.1}, we have the following lemma. Its proof is also
similar to the proof of Lemma \ref{xxlem3.1}, so is omitted.

\begin{lemma}
\label{xxlem3.3} Retain the notation as above. Let $2\leq s\leq n$.
\begin{enumerate}
\item[(1)]
$x_{i,s} x_{j,s}+x_{j,s} x_{i,s}=0$ for all $i<j$ and $i<s$.
\item[(2)]
$x_{i,s}=x_{i,s-1}$ if $i<s$ and 
$x_{i,s}^2=x_{i,s-1}^2-\frac{1}{4}a_{s-1i,s-1}^2 x_{s-1,s-1}^{-2}$ for all 
$i\geq s$.
\item[(3)]
$x_{i,s}x_{j,s}+x_{j,s}x_{i,s}=a_{ij,s-1}-\frac{1}{2} 
a_{s-1i,s-1}a_{s-1j,s-1} x_{s-1,s-1}^{-2}$ for all $s\leq i<j\leq n$.
\item[(4)]
Let $M_s$ be the matrix $( x_{i,s}x_{j,s}+x_{j,s}x_{i,s})_{1\leq i,j\leq n}$.
Then $\det M_s=\det M_1$.
\item[(5)]
Let $C_{s-1}$ be the Ore set 
$\{x_{1,1}^{2i_1}x_{2,2}^{2i_1}\cdots x_{s-1,s-1}^{2i_{s-1}}
\}_{i_1,\cdots,i_{s-1}\geq 1}$. Then the localization $\VA[C_{s-1}^{-1}]$
is free over $Z[C_{s-1}^{-1}]$ with basis $\{x_{1,s}^{d_1}\cdots x_{n,s}^{d_n}
\mid d_s=0,1\}$. 
\end{enumerate}
\end{lemma}

We need two more lemmas before we prove the main result.

\begin{lemma}
\label{xxlem3.4} Let $T$ be a commutative domain. Let $A$ be a
$T$-algebra containing $T$ as a subalgebra,  
generated by $x_1,\cdots, x_n$ and satisfying the relations
$x_j x_i+x_ix_j=0$ for all $i<j$ and $x_i^2=a_i\in T$. Suppose that 
$A$ is a free module over $T$ with basis $\{x_{1}^{d_1}\cdots x_{n}^{d_n}
\mid d_s=0,1\}$. Then 
$$d(A/T)=_{T^{\times}}
(\prod_{i=1}^{n} (2 x_i^2))^{2^{n-1}}=_{T^{\times}}
(\prod_{i=1}^{n} (x_i^2))^{2^{n-1}}.$$
\end{lemma}

\begin{proof} Let $B=T_{-1}[x_1,\cdots,x_n]$ and $Z=T[x_1^2,\cdots,x_n^2]$.
Then $B$ is a free module over $Z$ with basis $\{x_{1}^{d_1}\cdots x_{n}^{d_n}
\mid d_s=0,1\}$. Let $g$ be the algebra map from $B$ to $A$ sending
$T$ to $T$, $x_i$ to $x_i$. Then the hypotheses in
Lemma \ref{xxlem1.2} holds. By Lemma \ref{xxlem1.2},
$g(d(B/Z))=_{T^\times} d(A/T)$. Note that $d(B/Z)$ was computed in
Proposition \ref{xxpro1.4}(3) to be
$(\prod_{i=1}^{n} (2 x_i^2))^{2^{n-1}}$ as we assume that $2$ is invertible. 
Now the assertion follows.
\end{proof}

Let $A$ be an Ore domain and let $Q(A)$ denote the skew field of fractions of
$A$. Let $Z$ be the commutative subalgebra $T[x_1^2,\cdots,x_n^2]\subset 
\VA$. For each $1\leq 1\leq n$, let $Z_i$ be the subring of 
$Q(Z)$ of the form $Q(T[x_1^2,\cdots, \widehat{x_i^2}, \cdots, x_n^2])[x_i^2]$.

\begin{lemma}
\label{xxlem3.5}
Retain the above notation.
\begin{enumerate}
\item[(1)]
$\bigcap_{i=1}^n Z_i=Q(T)[x_1^2,\cdots,x_n^2]$.
\item[(2)]
$Z[C_{n-1}^{-1}]\subseteq Z_n$ where $Z[C_{n-1}^{-1}]$
is defined in Lemma \ref{xxlem3.3}(5).
\end{enumerate}
\end{lemma}

\begin{proof} (1) This is an easy commutative algebra fact.

(2) By Lemma \ref{xxlem3.3}(2) and induction, each $x_{i,s}^2$,
for all $1\leq i<n$ and all $1\leq s\leq n$, is in 
$Q(T[x_1^2,\cdots,x_{n-1}^2])$. So $Z[C_{n-1}^{-1}]\subseteq Z_n$.

\end{proof}

\begin{theorem}
\label{xxthm3.6} Suppose $2$ is invertible. Let
$Z=T[x_1^2,\cdots,x_n^2]$. Then 
$$d(\VA/Z)=_{T^{\times}}(\det M_1)^{2^{n-1}}$$
where $M_1$ is given in \eqref{E3.0.1}.
\end{theorem}

\begin{proof} Consider the variables $\{x_{i,n}\}_{i=1}^n$ defined in
Lemma \ref{xxlem3.3}. By Lemma \ref{xxlem3.3}(5), $\VA[C_{n-1}^{-1}]$
is free over $Z[C_{n-1}^{-1}]$ with basis $\{x_{1,s}^{d_1}\cdots x_{n,s}^{d_n}
\mid d_s=0,1\}$. By Lemma \ref{xxlem3.4}, the discriminant 
$d(\VA[C_{n-1}^{-1}]/Z[C_{n-1}^{-1}])$ is of the form 
$(\prod_{i=1}^{n} (x_i^2))^{2^{n-1}}$ up to a unit in 
$Z[C_{n-1}^{-1}]$. By Lemma \ref{xxlem3.3}(4), we have
$$d(\VA[C_{n-1}^{-1}]/Z[C_{n-1}^{-1}])=(\prod_{i=1}^{n} (x_i^2))^{2^{n-1}}
=(\det M_n)^{2^{n-1}}=(\det M_1)^{2^{n-1}}.$$
By Lemma \ref{xxlem1.3},
$$d(\VA/Z)=_{(Z[C_{n-1}^{-1}])^{\times}} 
d(\VA[C_{n-1}^{-1}]/Z[C_{n-1}^{-1}])=_{(Z[C_{n-1}^{-1}])^{\times}} 
(\det M_1)^{2^{n-1}}.$$
Let $\Phi$ be the element $d(\VA/Z)^{-1}(\det M_1)^{2^{n-1}}$.
Then $\Phi\in (Z[C_{n-1}^{-1}])^{\times}$. This means that both
$\Phi$ and $\Phi^{-1}$ are in $Z[C_{n-1}^{-1}]\subseteq Z_n$.
By symmetry, $\Phi$ is $Z_i$ for all $i$. Thus $\Phi\in \bigcap_{i=1}^n 
Z_i=Q(T)[x_1^2,\cdots,x_n^2]$. Similarly, $\Phi^{-1}$ is in
$Q(T)[x_1^2,\cdots,x_n^2]$. Therefore $\Phi, \Phi^{-1}\in Q(T)$. 

Write $d(\VA/Z)=c \; (\det M_1)^{2^{n-1}}$ where $c=\Phi^{-1}\in Q(T)$. 
It remains to show $c\in Z^{\times}$. Note that
$\VA$ is a filtered algebra such that $\gr \VA\cong 
T_{-1}[x_1,\cdots,x_n]$. By Lemma \ref{xxlem1.5}, 
$$\gr d(\VA/Z)=_{Z^{\times}} d(\gr \VA/\gr Z).$$
The left-hand side of the above is $c\; (\prod_{i=1}(x_i^2))^{2^{n-1}}$
and the right-hand side of the above is 
$(\prod_{i=1}(x_i^2))^{2^{n-1}}$ by Proposition \ref{xxpro1.4}(3)
(assuming $2$ is invertible). Thus $c\in Z^{\times}$ as required.
\end{proof}

Theorem \ref{xxthm0.2} is a special case of Theorem 
\ref{xxthm3.6} by taking $a_{ij}=1$ for all $i<j$.

The algebras $V_n({\mathcal A})$ and $W_n$ are special Clifford algebras.
Now we consider a Clifford algebra in a more general setting. Let $T$ 
be a commutative domain and $V$ be a free $T$-module of rank $n$. Given 
a quadratic form $q:V \longrightarrow T$, we can associate to this data 
the Clifford algebra
\begin{eqnarray*}
C ( V,q ) & = & \frac{T \langle V \rangle}{( x^{2}-q(x)\mid x \in V )}.
\end{eqnarray*}
Note that this $q$ is different from the parameter $q$ in the definition
of the $q$-quantum Weyl algebra $A_q$ and the parameter set ${\bf q}$ in 
the $V_n({\bf q},{\mathcal A})$ and $T_{\bf q}[x_1,\cdots,x_n]$. 
Consider the bilinear form associated to $q$,
\begin{equation}
\label{E3.6.1}\tag{E3.6.1}
b(x,y)=\frac{1}{2}(q(x+y)-q(x)-q(y))
\end{equation}
for all $x,y\in V$. If we choose a 
$T$-basis $x_{1} , \ldots ,x_{n}$ for $V$ and let 
\begin{equation}
\label{E3.6.2}\tag{E3.6.2}
{\mathfrak B} \assign (b_{ij})=(b(x_i,x_j) )_{n \times n} \in T^{n \times n}
\end{equation}
be the symmetric matrix which represents $b$ with respect to this basis, 
then the relations of $C(V,q)$ are
\begin{equation}
\label{E3.6.3}\tag{E3.6.3}
x_{i} x_{j} +x_{j} x_{i} = 2b_{i j}, \quad {\text{for all}}\; i,j.
\end{equation}
Define $\det (q)$ to be $\det ({\mathfrak B})$. 

The following main result is a consequence of Theorem \ref{xxthm3.6}
and Lemma \ref{xxlem1.2}.

\begin{theorem}
\label{xxthm3.7} 
Let $A:=C(V,q)$ be a Clifford algebra over a commutative domain $T$
defined by a quadratic form $q: V\to T$. Pick a $T$-basis of $V$, 
say $\{x_i\}_{i=1}^n$. Then 
\begin{equation}
\label{E3.7.1}\tag{E3.7.1}
d(A/T)=_{T^\times} (\det (x_ix_j+x_jx_i)_{n\times n})^{2^{n-1}}
=_{T^\times} \det (q)^{2^{n-1}}.
\end{equation}
\end{theorem}

\begin{proof} Let $b: V^{\otimes 2}\to T$ be the symmetric bilinear 
form associated to the quadratic form $q$. Let $a_{ij}=2 b(x_i,x_j)$
for all $i<j$ and ${\mathcal A}=\{a_{ij}\}_{1\leq i<j\leq n}$. 
Then there is a canonical algebra surjection
$\pi: V_n({\mathcal A})\to C(V,q)$ sending $x_i\to x_i$ for all 
$i=1,\cdots, n$ and $t\to t$ for all $t\in T$, and the kernel of 
the $\pi$ is the ideal generated by
$\{x_i^2-b_{ii}\}_{i=1}^n$. Clearly, $\pi(T[x_1^2,\cdots,x_n^2])=T$ 
and the matrix $(x_ix_j+x_jx_i)_{n\times n}$ equals $M_1$. 
It is easy to check that 
$\{x_1^{d_1}\cdots x_n^{d_n} \mid d_i=0,1\}$ is a basis of 
$V_n({\mathcal A})$ over $T[x_1^2,\cdots,x_n^2]$ and 
a basis of $C(V,q)$ over $T$.
The first equation of \eqref{E3.7.1} follows from Theorem \ref{xxthm3.6}
and Lemma \ref{xxlem1.2} and the second equation
follows from the fact that $2{\mathfrak B}=(x_ix_j+x_jx_i)_{n\times n}$
and that $2$ is invertible.
\end{proof}

In the rest of this section we briefly discuss ``generic Clifford 
algebras'' which will appear again in Section \ref{xxsec8}. (This generic
Clifford algebra should be called a ``universal Clifford algebra'', but the 
term ``universal Clifford algebra'' has already been used).

Fix an integer $n$. Let $I$ be the set $\{(i,j)\mid 1\leq i\leq
j\leq n\}$ that can be thought as the quotient set $\{(i,j)\mid 
1\leq i,j\leq n\}/((i,j)\sim (j,i))$. Let $w$ denote the integer 
$\frac{1}{2} n(n+1)$. There is a bijection between $I$ and the set 
of first $w$ integers $\{1,2,\cdots, w\}$. Let $T_g$ be the 
commutative domain $k[t_{(i,j)}\mid (i,j)\in I]$
which is isomorphic to $k[t_1,\cdots,t_w]$. Define a $T_g$-algebra
$A_g$ generated by $\{x_1,\cdots,x_n\}$ and subject to the relations
\begin{equation}
\label{E3.7.2}\tag{E3.7.2}
x_i x_j+x_jx_i= 2 t_{(i,j)}, \quad \forall\; 1\leq i\leq j\leq n.
\end{equation}
Let $V_{g}=\bigoplus_{i=1}^n T_g x_i$. Define a bilinear form
$b_g: V_{g}\otimes V_{g}\to T_g$ by $b_g(x_i,x_j)=t_{(i,j)}$ and 
the associated quadratic form by $q_g(x)= b_g(x,x)$ for all 
$x\in V_{g}$. The ``generic Clifford algebra'' $A_g$ is defined 
to be the Clifford algebra associated to $(V_{g},q_g)$. For any 
Clifford algebra $C(V,q)$ over a commutative ring $T$, by 
comparing \eqref{E3.6.3} with \eqref{E3.7.2}, one sees
that there is an algebra map $A_g\to C(V,q)$ sending $x_i\to x_i$ 
and $t_{(i,j)} \to b_{ij}$. Define $\deg x_i=1$ for all $i$ and 
$\deg t_{(i,j)}=2$ for all $(i,j)\in I$. Then $A_g$ is a connected 
graded algebra over $k$.

We also define some factor algebras of $A_g$.
Let $J$ be a subset of $\{(i,j)\mid 1\leq i<j \leq n\}$ and 
$w_J$ denote the integer $w-|J|$. Let $T_{g,J}$ be the 
commutative polynomial ring $k[t_{i,j}\mid (i,j)\in I\setminus J]$,
which is isomorphic to $k[t_1,\cdots,t_{w_J}]$. Define a $T_{g,J}$-algebra
$A_{g,J}$ generated by $\{x_1,\cdots,x_n\}$ and subject to the relations
\begin{equation}
\label{E3.7.3}\tag{E3.7.3}
x_i x_j+x_jx_i= \begin{cases} 2 t_{(i,j)}, & (i,j)\in I\setminus J,\\
0, & (i,j)\in J.\end{cases}
\end{equation}
Let $V_{g,J}=\bigoplus_{i=1}^n T_{g,J} x_i$. Define a bilinear form
$b_{g,J}: V_{g,J}\otimes V_{g,J}\to T_{g,J}$ by $b_{g,J}(x_i,x_j)=
\begin{cases} t_{(i,j)} & (i,j)\in I\setminus J,\\0, & (i,j)\in J.
\end{cases}\;$ 
and the associated quadratic form by $q_{g,J}(x)= b_g(x,x)$ for all 
$x\in V_{g,J}$. Then $A_{g,J}$ is the Clifford algebra associated to 
$(V_{g,J},q_{g,J})$. If $J\subseteq J'\subseteq \{(i,j)\mid 1\leq i<j 
\leq n\}$, there is an algebra map
$A_{g,J}\to A_{g,J'}$ sending $x_i\to x_i$ and $t_{(i,j)}\to 
\begin{cases} t_{(i,j)}, & (i,j)\notin J',\\ 0 & (i,j)\in J'\setminus J.
\end{cases}$. In particular, $A_{g,J}$ is a connected graded
factor ring of $A_g$.

In part (4) of the next lemma, we will use a few undefined concepts 
that are related to the homological properties of an algebra. We refer to 
\cite{Le, LP, RZ} for definitions.

\begin{lemma}
\label{xxlem3.8} Retain the above notation. Assume that $k$ is a field
of characteristic not two. Let $J'$ be subset of $\{(i,j)\mid 1\leq i<j 
\leq n\}$ and $J=J'\setminus \{(i_0,j_0)\}$ for some 
$(i_0,j_0)\in J'$.
\begin{enumerate}
\item[(1)]
The Hilbert series of $A_g$ is
$$H_{A_g}(t)=\frac{(1+t)^n}{(1-t^2)^w}$$
where $w=\frac{1}{2} n(n+1)$.
\item[(2)]
The Hilbert series of $A_{g,J}$ is
$$H_{A_{g,J}}(t)=\frac{(1+t)^n}{(1-t^2)^{w_J}}$$
where $w_J=w-|J|$.
\item[(3)]
$t_{(i_0,j_0)}$ is a central regular element in $A_{g,J'}$, 
and $A_{g,J}=A_{g,J'}/(t_{(i_0,j_0)})$.
\item[(4)]
$A_g$ and $A_{g,J}$ are connected graded Artin-Schelter regular, Auslander regular, 
Cohen-Macaulay noetherian domains.
\end{enumerate}
\end{lemma}

\begin{proof} (1) Note that $A_{g}$ is a free module over $T_{g}$ 
with basis $\{x_1^{d_1} \cdots x_n^{d_n}\mid d_s=0,1\}$. Recall
$\deg x_i=1$ and $\deg t_{(i,j)}=2$. We have
$$H_{A_{g}}(t)=(1+t)^n H_{T_{g}}(t)=\frac{(1+t)^n}{(1-t^2)^w}.$$

(2) The proof is similar. Use the fact 
$H_{T_{g,J}}(t)=\frac{1}{(1-t^2)^{w_J}}$.

(3) It is clear that $t_{(i_0,j_0)}$ is central in $A_{g,J'}$ 
and $A_{g,J}=A_{g,J'}/(t_{(i_0,j_0)})$. So the ideal
$(t_{(i_0,j_0)})$ is the left ideal $t_{(i_0,j_0)}A_{g,J'}$ 
and the right ideal $A_{g,J'}t_{(i_0,j_0)}$. By parts (1) and (2),
the Hilbert series of $(t_{(i_0,j_0)})$ is $t^2
H_{A_{g,J'}}(t)$. So $t_{(i_0,j_0)}$ is regular. 

(4) We only provide a proof for $A_{g}$. The proof 
for $A_{g,J}$ is similar.

From part (3), $J_M:=\{t_{(i,j)}\mid 1\leq i<j \leq n\}$ 
is a sequence of regular central elements in $A_g$ of positive degree. 
It is easy to see that $A_{g,J_M}(=A_{g}/(J_M))$ is isomorphic 
to the skew polynomial ring $k_{-1}[x_1,\cdots,x_n]$, which is 
an Artin-Schelter regular, Auslander regular, Cohen-Macaulay 
noetherian domain. Applying \cite[Lemma 7.6]{LP} repeatedly, 
$A_{g}$ has finite global dimension. Applying 
\cite[Proposition 3.5, Theorem 5.10]{Le} repeatedly,
$A_g$ is a noetherian Auslander Gorenstein and Cohen-Macaulay
domain. By \cite[Theorem 6.3]{Le}, $A_g$ is Artin-Schelter Gorenstein.
Since $A_g$ has finite global dimension, it is Auslander
regular and Artin-Schelter regular. 
\end{proof}

\begin{remark}
\label{xxrem3.9} Retain the above notation.
\begin{enumerate}
\item[(1)]
Some homological properties of the algebra $A_g$ are given in 
Lemma \ref{xxlem3.8}. It would be interesting to work out 
combinatorial and geometric invariants (and properties) of $A_g$. For 
example, what are the point-module and line-module schemes 
of $A_g$? Definitions of these schemes 
can be found in \cite{VVR, VVRW}.
\item[(2)]
Another way of presenting $A_g$ is the following. Let $S$ be a 
$k$-vector space of dimension $n$. Define $A_g$ to be
$k\langle S\rangle /([x^2,y]=0 \mid \forall \; x,y, \in S)$.
By using this new expression, one can easily see that the group
of graded algebra automorphisms of $A_g$, denoted by $\Aut_{gr}(A_g)$, 
is isomorphic to ${\rm{GL}}_n(k)$. 
\item[(3)]
Suppose $n\geq 2$.
The full automorphism group $\Aut(A_g)$ has not been determined.
It is known that $\Aut(A_g)$ is not affine. For example, if $f(t)$
is a polynomial in $t$, then 
$$x_i\to \begin{cases} x_i & i>1,\\
x_1+f([x_1,x_2]^2) x_2 & i=1, \end{cases}$$
extends to an algebra automorphism of $A_g$.
\item[(4)]
It seems interesting to study ``cubic-algebra''
$k\langle S\rangle /([x^3,y]=0\mid \forall\; x,y\in S)$
and higher-degree analogues.
\item[(5)]
The quotient division ring of $A_g$, denoted by $D_g$, is called the
``generic Clifford division algebra of rank $n$''. 
It would be interesting to study algebraic properties or invariants 
of $D_g$.
\end{enumerate}
\end{remark}

\section{Center of skew polynomial rings}
\label{xxsec4}

To use the discriminant most effectively, one needs to first understand
the center of an algebra. In this section we give a criterion for 
when $T_{\bf q}[x_1,\cdots,x_n]$ is free over its center and when 
the center of $T_{\bf q}[x_1,\cdots,x_n]$ is a polynomial ring. 

Recall that $T$ is a commutative domain and 
${\bf q}: =\{q_{ij}\in T^\times \mid 1\leq i <j\leq n\}$ 
is a set of invertible scalars. Let $P: =T_{\bf q}[x_{1},\ldots ,x_{n}]$ 
be the skew polynomial ring over $T$ subject to the relations 
\eqref{E0.2.1}. We assume that $d_{i j}:=o(q_{i j})<\infty$ and write 

\begin{equation}
\label{E4.0.1}\tag{E4.0.1}
q_{i j} = \exp ( 2 \pi\sqrt{-1}\; k_{i j} /d_{i j} ),
\end{equation}
where $|k_{ij}|<d_{ij}$ and $( k_{ij} ,d_{i j} ) =1$. Note that, by our 
convention, $q_{ij}=q_{ji}^{-1}$ for all $i,j$. Hence, we choose 
$k_{i j} =-k_{j i}$ and $d_{i j} =d_{j i}$. We also adopt the convention 
that if $q_{i j} =1$ then $k_{i j} =0$ and $d_{i j} =1$. In particular, 
$k_{i i} =0$ and $d_{i i} =1$. We can extend $P$ to 
$P[x_1^{-1}, ..., x_n^{-1}]$, with an inverse for 
each $x_i$, with the following expected relations
$$x_i x_i^{-1} = x_i^{-1} x_i = 1, \;
x_j x_i^{-1} = q_{ij}^{-1} x_i^{-1} x_j,\quad
{\text{and}}\quad x_j^{-1} x_i^{-1} = q_{ij} x_i^{-1} x_j^{-1}.$$ 
We need to do some analysis to understand the center of $P$.
Let $\eta_{i}$ denote conjugation by $x_{i}$, sending 
$f\longmapsto x_i^{-1} f x_i$, and let 
$\xi =x_{1}^{s_{1}} \cdots x_{n}^{s_{n}}$. Then
\begin{eqnarray*}
  \eta_{i} ( \xi ) & = & \exp 
( 2 \pi \sqrt{-1}\; \mathbf{e}_{i}^T Y \mathbf{s} ) \; \xi
\end{eqnarray*}
where $Y \in \mathfrak{s o}_{n} ( \mathbb{Q} )$ whose $( i,j )$-th entry is
$k_{ij} /d_{ij}$, $\mathbf{s}$ is the column vector whose $i$-th entry is 
$s_i$ appearing in the powers of $\xi$, and $\mathbf{e}_{i}$ the $i$-th 
standard basis vector in $\mathbb{Q}^{n}$. 

\begin{lemma}
\label{xxlem4.1} Retain the above notation. Then 
$\xi$ is in the center $Z(P)$ of $P$ if and only if $Y \mathbf{s} \in
\mathbb{Z}^{n}$.
\end{lemma}

\begin{proof} Since $P$ is generated by $\{x_i\}$, $\xi\in Z(P)$ if and
only if $\eta_i(\xi)=\xi$ for all $i$, if and only if $\exp 
( 2 \pi \sqrt{-1}\; \mathbf{e}_{i}^T Y \mathbf{s} )=1$, if and only if 
$\mathbf{e}_{i}^T Y \mathbf{s}\in {\mathbb Z}$ for all $i$, and finally, 
if and only if $Y \mathbf{s} \in \mathbb{Z}^{n}$.
\end{proof}

By choosing the standard basis for $\mathbb{Q}^{n}$, we can consider $Y$ 
as a linear transformation $\mathbb{Q}^{n} \longrightarrow \mathbb{Q}^{n}$ 
by sending $\mathbf{s}\longmapsto Y\mathbf{s}$. Here we view 
$\mathbb{Q}^{n}$ as column vectors and $Y$ as a left multiplication. 
We can restrict this map to $\mathbb{Z}^{n} \subset \mathbb{Q}^{n}$ 
(embedded via the standard basis) and compose with the quotient 
$\mathbb{Q}^{n} \longrightarrow \mathbb{Q}^{n} /\mathbb{Z}^{n}$ to obtain 
a $\mathbb{Z}$-module homomorphism $Y' :\mathbb{Z}^{n} \longrightarrow
\mathbb{Q}^{n} /\mathbb{Z}^{n}$. 

\begin{lemma}
\label{xxlem4.2} Retain the above notation.
Then $\xi \in Z ( P )$ if and only if $\mathbf{s} \in \ker ( Y' )$.
\end{lemma}

\begin{proof} By lemma \ref{xxlem4.1}, $\xi \in Z ( P )$ if and only if
$Y\mathbf{s}\in \mathbb{Z}^{n}$, which is equivalent to 
$Y'(\mathbf{s})=0$ by the definition of $Y'$.
\end{proof}

Let $D$ be the matrix $(d_{ij})_{n\times n}$ and let $L_i$ be the 
$\mathrm{lcm}$ of the entries in $i$-th row of $D$, namely,
$L_{i} = \mathrm{lcm} \{d_{i j} \mid j=1, \ldots ,n \}$. 
Since $D$ is a symmetric matrix, $L_i$ is also the $\mathrm{lcm}$ of the 
entries in $i$-th column. Observe that $Z ( P )$ contains the central 
subring $P' \assign k [x_{1}^{L_{1}} , \ldots ,x_{n}^{L_{n}} ]$.
In other words, $\ker ( Y' )$ contains the $\mathbb{Z}$-lattice 
$\Lambda$ spanned by $L_{i} \mathbf{e}_{i}$ for $i=1, \ldots ,n$. 
Therefore $Y'$ factors through 
$$\mathbb{Z}^{n} \longrightarrow M \assign \mathbb{Z}^{n} / \Lambda
=\bigoplus_{i=1}^n {\mathbb Z}/L_i{\mathbb Z}.$$ 
For each $\mathbf{s}\in {\mathbb Z}^n$, the $i$-th entry of $Y'(\mathbf{s})$ 
is $\sum_{j} k_{ij}s_j/d_{ij}\in {\mathbb Q}/{\mathbb Z}$, which is
$L_i$-torsion, or equivalently, in $L_i^{-1}{\mathbb Z}/{\mathbb Z}$. 
Therefore $Y'$ induces a map
$$M\longrightarrow M':= \bigoplus_{i=1}^n L_i^{-1}{\mathbb Z}/{\mathbb Z}.$$
Since $M'$ is naturally isomorphic to $M$, we can define an endomorphism 
$$\overline{Y}: M \longrightarrow M$$ by setting 
$$\overline{Y} \mathbf{s}=(\sum_{j=1}^n L_i (k_{ij}s_j/d_{ij}))_{i=1}^n.$$
In particular,
$\overline{Y} \mathbf{e}_j = \sum_{i=1}^n (k_{ij} L_i / d_{ij}) \mathbf{e}_i$. 
Sometimes we think of $\overline{Y}$ as a matrix
$$\overline{Y}=(k_{ij}L_i/d_{ij})_{n\times n}={\text{diag}} (L_1,\cdots,L_n) Y.$$

The following lemma is a re-interpretation of 
\cite[Lemma 2.3]{CPWZ2}.

\begin{lemma}
\label{xxlem4.3}
Retain the above notation. The following are equivalent.
\begin{enumerate}
\item[(1)]
The center $Z ( P )$ of $P$ is a polynomial ring.
\item[(2)]
$Z(P) = P'$.
\item[(3)]
$\ker (\overline{Y}) =0$. 
\item[(4)]
$\overline{Y}$ is an isomorphism.
\end{enumerate}
\end{lemma}

\begin{proof} $(1) \Leftrightarrow (2)$: One implication is clear. 
For the other implication, we assume that the center $Z ( P )$ is 
a polynomial ring. By \cite[Lemma 2.3]{CPWZ2}, $Z(P)$ is of the
form $T[x_1^{a_i},\cdots,x_n^{a_i}]$. It is easy to check that 
$L_i\mid a_i$ for all $i$. Since $Z(P)\supseteq P'$, $a_i=L_i$ for 
all $i$. The assertion follows.

$(3) \Longrightarrow(2)$:
Let $\xi:=x_1^{s_1} \cdots x_n^{s_n} \in Z(P)$ and let
$\mathbf{s}=(s_i)_{i=1}^n$. By Lemma
\ref{xxlem4.2}, $\mathbf{s}\in \ker (Y')$. Since $\overline{Y}$ 
is induced by $Y'$, $\overline{Y}(\mathbf{s})=0$. By part (3), 
$\mathbf{s}=0$ in $M={\mathbb Z}^n/\Lambda$. So $\mathbf{s}
\in \Lambda$, which is equivalent to $\xi\in P'$. Therefore, 
$Z(P) = P'$ as desired.

$(2) \Longrightarrow(3)$:
Let $\xi:=x_1^{s_1} \cdots x_n^{s_n}\in P$ where 
$\mathbf{s}:=(s_i)_{i=1}^n\in \ker (\overline{Y})$ viewing as
a vector in $M$. By the definition of $M$, we might assume that
each $s_i$ is non-negative and less than $L_i$. Since $\overline{Y}$ is
induced by $Y'$, we have that $\mathbf{s}\in \ker (Y')$. 
By Lemma \ref{xxlem4.2}, $\xi\in Z(P)$. By part (2) and our choice 
of $0\leq s_i<L_i$, $\xi=1$ or $\mathbf{s}=0$ as desired.

$(3) \Leftrightarrow (4)$: This is clear since $M$ is finite.
\end{proof}

The advantage of working with $\overline{Y}$ is that $\ker(\overline{Y})=0$
is equivalent to $\overline{Y}$ being an isomorphism.
Next we need to understand when $\overline{Y}$ is an isomorphism. For the
rest of this section we use $\otimes $ for $\otimes_{\mathbb Z}$ and 
${\mathbb F}_p$ for $\mathbb{Z}/p\mathbb{Z}$.

\begin{lemma}
\label{xxlem4.4}
The morphism $\overline{Y}$ is an isomorphism if and only if $\overline{Y} \otimes
\mathbb{F}_p$ is an isomorphism for all primes $p$. 
\end{lemma}

\begin{proof}
As a $\mathbb{Z}$-module, $M$ is finite, and it suffices to show that 
$\overline{Y}$ is surjective if and only if $\overline{Y} \otimes 
\mathbb{F}_p$ is surjective for each prime $p$. This is clear 
since $- \otimes \mathbb{F}_p$ is right exact, so surjectivity 
of a map can be checked on closed fibers.
\end{proof}

Fix any prime $p$. Let $M_p = M \otimes \mathbb{F}_p$, and 
$\overline{Y}_p = \overline{Y} \otimes \mathbb{F}_p$. For any 
$\mathbf{e}_i$, if $L_i \notin p\mathbb{Z}$, then the image of 
$\mathbf{e}_i$ is zero in $M_p$. We can therefore use 
$\{\mathbf{e}_i | L_i \in p\mathbb{Z}\}$ as a basis of $M_p$.
Consequently, $M_p$ is a vector space over ${\mathbb F}_p$ of dimension 
at most $n$, and we can write $\overline{Y}_p$ as a matrix over ${\mathbb F}_p$. 
Next we will decompose the vector space $M_p$ and the matrix $\overline{Y}_p$.

For each positive integer $m$, let $M_{p,m}$ denote the subspace of $M_p$ 
generated by $\{\mathbf{e}_i | L_i \in p^m \mathbb{Z} - p^{m+1} \mathbb{Z}\}$. 
Let $\overline{Y}_{p,m}$ be the endomorphism 
$$M_{p,m} \longrightarrow M_p \stackrel{\overline{Y}_p}{\longrightarrow} M_p 
\longrightarrow M_{p,m}$$ 
where the first map is the inclusion and the last map is the natural 
projection using the given basis $\{\mathbf{e}_i | L_i \in p\mathbb{Z}\}$.
Then $\overline{Y}_{p,m}$ can be expressed as the submatrix of $\overline{Y}$ taken 
from the row and columns with indices $i$ such that $\mathbf{e}_i \in M_{p,m}$. 
For all but finitely many values of $m$, $M_{p,m} = 0$, and in this case, 
$\overline{Y}_{p,m}$ is a $0 \times 0$ matrix. We adopt the convention that 
the determinant of a $0 \times 0$ matrix is $1$. In general, 
$\det(\overline{Y}_{p,m})$ is in ${\mathbb F}_p$.

\begin{lemma}
\label{xxlem4.5} The following are equivalent.
\begin{enumerate}
\item[(1)]
The map $\overline{Y}_p$ is an isomorphism.
\item[(2)]
For all positive integers $m$, $\overline{Y}_{p,m}$ is an isomorphism.
\item[(3)]
$\det ( \overline{Y}_{p,m } ) \neq 0$ for all positive integers $m$. 
\end{enumerate}
\end{lemma}

\begin{proof} It is clear that (2) and (3) are equivalent, so we need only
show that (1) and (2) are equivalent. 

Let $m > 0$, and let $i, j$ be such that $L_i \in p^m \mathbb{Z} - p^{m+1} 
\mathbb{Z}$ and $L_j \notin p^m \mathbb{Z}$. Since $L_j = \mathrm{lcm} 
\{d_{kj} \mid k=1,...,n\}$, we have $d_{ij} \notin p^m \mathbb{Z}$, and 
$k_{ij} L_i / d_{ij} \in p \mathbb{Z}$. Therefore, the 
$\mathbf{e}_i$-component of $\overline{Y}_p \mathbf{e}_j$ is zero. We can 
extended this to show that for any $m > m' > 0$, the $M_{p,m'}$-component 
of $\overline{Y}_p (M_{p,m})$ is zero, or equivalently, 
$$\overline{Y}_p (M_{p,m})\subseteq \bigoplus_{n\geq m} M_{p,n}=:N_m.$$ 
This implies that, for any $m > 0$, $\overline{Y}_p$ acts as an endomorphism on 
$N_m$. Since each $M_p$ is finite dimensional, $\overline{Y}_p$ is an isomorphism 
if and only if it acts as an isomorphism on each subquotient 
$N_m / N_{m+1} \cong M_{p,m}$. This action is already given by 
$\overline{Y}_{p,m}$, so the assertion follows.
\end{proof}

Combining all the lemmas in this section we have

\begin{theorem}
\label{xxthm4.6} The center of the skew polynomial ring  
$T_{\bf q}[x_1,\cdots,x_n]$ is a polynomial ring if and only if 
$\det (\overline{Y}_{p,m})\neq 0$ for all primes $p$ and all 
integers $m>0$.
\end{theorem}

Theorem \ref{xxthm4.6} is a slight generalization of Theorem
\ref{xxthm0.3}(a) without the hypothesis that $q_{ij}\neq 1$ for all
$i\neq j$. The definition of the matrices $\overline{Y}_{p,m}$ is not 
straightforward, so we give an example below. Hopefully, this 
example will show that this matrix is not hard to understand.

\begin{example}
\label{xxex4.7} 
We start with the following skew-symmetric
matrix with entries in ${\mathbb Q}$
$$Y := 
\begin{pmatrix}
0 &  4/27 &  2/9 &     0 &          2/3 &           3/5 \\
-4/27 & 0 &    1/3 & 7/9 &           1/3 &           1/5 \\
-2/9 &  -1/3 &  0 &  1/6 &           1/2 &           1/2 \\
0 & -7/9 &  -1/6 &  0 &             2/3 &           0 \\
-2/3 &  -1/3 &  -1/2 &  -2/3 &  0 &             5/8 \\
-3/5 &  -1/5 &  -1/2 &  0 &             -5/8 &  0 \\
\end{pmatrix}.
$$ 
One can easily construct $q_{ij}$ by \eqref{E4.0.1} and the skew 
polynomial ring $T_{\bf q}[x_1,\cdots,x_6]$ by \eqref{E0.2.1}, but 
the point of this example is to work out
the matrices $\overline{Y}_{p,m}$ for all primes $p$ and all $m>0$. 
By considering the denominators of the entries of $Y$, one sees that
$$(L_1, L_2, L_3, L_4, L_5, L_6) = 
(3^3 \cdot 5,\; 3^3 \cdot 5,\; 2 \cdot 3^2, \;2 \cdot 3^2,\; 2^3 \cdot 3, \;
2^3 \cdot 5).$$
This implies that $\overline{Y}_{p,m}$ is a trivial matrix (or a 
$0\times 0$-matrix) except for $p=2,3,5$. Next we consider

$$\overline{Y} = {\text{diag}} (L_1,\cdots,L_6) Y=
\begin{pmatrix}
0 &20 &            30 &            0 &             90 & 81 \\
-20 & 0 &         45 &            105 &           45 &27 \\
-4 &  -6 &     0 &             3 &             9 &9 \\
0 &  -14 &    -3 &            0 &             12 &0 \\
-16 &  -8 & -12 &           -16 &           0 &15 \\
-24 &  -8 & -20 &           0 &             -25 &   0 \\
\end{pmatrix}.
$$ 
Recall that $M_{p,m}$ has a basis $\{\mathbf{e}_i\mid L_i\in p^m{\mathbb Z}-
p^{m+1}{\mathbb Z}\}$ and $\overline{Y}_{p.m}$ is the square submatrix of
$\overline{Y}$ with indices $\{i\mid L_i\in p^m{\mathbb Z}-p^{m+1}{\mathbb Z}\}$
and with entries evaluated in ${\mathbb F}_p$.

For $p=2$, $\overline{Y}_{2,m}$ are the following:
\begin{enumerate}
\item[]
$\overline{Y}_{2,1}$ 
is the principle $(3,4)$-submatrix of $Y$, and 
$$\overline{Y}_{2,1} = 
\begin{pmatrix}
0 &             1\\
1 &             0\\
\end{pmatrix}.
$$
\item[]
$\overline{Y}_{2,3}$ uses indices $5, 6$, and 
$$\overline{Y}_{2,3} = 
\begin{pmatrix}
0 &             1\\
1 &             0\\
\end{pmatrix}.
$$
\item[]
For all $m=2$ or $m>3$, $\overline{Y}_{2,m}$ is trivial. 
\end{enumerate}
Therefore, $\overline{Y}_2$ is an isomorphism by Lemma \ref{xxlem4.5}. 

For $p=3$, $\overline{Y}_{3,m}$ are the following:
\begin{enumerate}
\item[]
$\overline{Y}_{3,1}$ uses only index $5$, and is the $1 \times 1$ zero matrix. 
\item[]
$\overline{Y}_{3,2}$ uses indices $3, 4$, and is the $2 \times 2$ zero matrix.
\item[]
$\overline{Y}_{3,3}$ uses indices $1, 2$, and 
$$\overline{Y}_{3,3} = 
\begin{pmatrix}
0 &     1\\
-1 &    0\\
\end{pmatrix}.
$$
\item[]
For all $m>3$, $\overline{Y}_{3,m}$ is trivial.
\end{enumerate}
Since $\det(\overline{Y}_{3,1}) = \det(\overline{Y}_{3,2}) = 0$, $\overline{Y}_3$ 
is not an isomorphism by Lemma \ref{xxlem4.5}. Consequently, the center of 
$T_{\bf q}[x_1,\cdots,x_6]$ is not a polynomial ring by Theorem \ref{xxthm4.6}.

For $p=5$, $\overline{Y}_{5,m}$ are the following:

\begin{enumerate}
\item[]
$\overline{Y}_{5,1}$ uses indices $1, 2, 6$, and 
$$\overline{Y}_{5,1} = 
\begin{pmatrix}
0 &     0 &     1\\
0 &     0 &     2\\
-1 &    -2 &    0\\
\end{pmatrix}.
$$
\item[]
For all $m>1$, $\overline{Y}_{5,m}$ is trivial.
\end{enumerate}
It is easy to check that $\det (\overline{Y}_{5,1})=0$. Therefore $\overline{Y}_{5}$
is not an isomorphism. 

For $p>5$, $\overline{Y}_{p,m}$ is trivial for all $m>0$.
\end{example}

\section{Low dimensional cases}
\label{xxsec5}

We start with some easy consequences of Theorem \ref{xxthm4.6} and then 
discuss the case when $n$ is 3 or 4.

\begin{corollary}
\label{xxcor5.1}
Suppose there are a prime $p$ and an $m>0$ such that $M_{p,m}$ is 
odd dimensional. Then $\overline{Y}_p$ is not an isomorphism. As a consequence, 
the center of $T_{\bf q}[x_1,\cdots,x_n]$ is not a polynomial ring.
\end{corollary}

\begin{proof}
If $\overline{Y}_{p,m}$ is a skew-symmetric matrix of odd size, its
determinant is zero (this is true even when $p=2$). The rest follows 
from Lemma \ref{xxlem4.5}
and Theorem \ref{xxthm4.6}.
\end{proof}

\begin{corollary}
\label{xxcor5.2}
Suppose there is a prime $p$ such that $M_p$ is odd dimensional.
Then $\overline{Y}_p$ is not an isomorphism. As a consequence, 
the center of $T_{\bf q}[x_1,\cdots,x_n]$ is not a polynomial ring.
\end{corollary}

\begin{proof}
Since $M_p = \bigoplus_{m=1}^\infty M_{p,m}$, if it is odd dimensional, at 
least one $M_{p,m}$ must be odd dimensional. The assertion follows
from Corollary \ref{xxcor5.1}.
\end{proof}

\begin{corollary}
\label{xxcor5.3}
Suppose, for each prime $p$, $p\mid  d_{ij}$ for at most one pair
$(i,j)$, $1\leq i<j\leq n$. Then $\overline{Y}_p$ is an isomorphism for 
each $p$. As a 
consequence, the center of $T_{\bf q}[x_1,\cdots,x_n]$ is a polynomial ring.
\end{corollary}

\begin{proof}
If each $d_{ij} \notin p \mathbb{Z}$, then each $L_i \notin p \mathbb{Z}$, 
$M_p = 0$ and $\overline{Y}_p$ is trivially an isomorphism.

If $d_{ij} \in p^m \mathbb{Z} - p^{m+1} \mathbb{Z}$ for some $i, j$ and some 
positive integer $m$, and each of every other term $d_{k \ell} \notin p 
\mathbb{Z}$, then $L_i, L_j \in p^m \mathbb{Z} - p^{m+1} \mathbb{Z}$, and 
each of every other $L_k \notin p \mathbb{Z}$. This shows that 
$\overline{Y}_{p,m}$ is a nonzero $2 \times 2$ skew-symmetric matrix 
(i.e. $\det(\overline{Y}_{p,m}) \neq 0$) and $M_{p,m'}=0$ for each $m' \neq m$. 
The rest follows from Lemma \ref{xxlem4.5}
and Theorem \ref{xxthm4.6}.
\end{proof}

Next we give simple criteria for $\overline{Y}$ to be an isomorphism in the 
cases where $n=3,4$. 

\begin{corollary}
\label{xxcor5.4}
The center of $T_{\bf q}[x_1,x_2,x_3]$ is a polynomial ring 
if and only if $(d_{i j},d_{i k}) =1$ for all different $i,j,k$. 
\end{corollary}

\begin{proof}
There are only three $d$ terms -- $d_{12}$, $d_{13}$, and $d_{23}$. If 
each $(d_{ij}, d_{ik}) = 1$, then no prime is a factor of more than one 
term in $\{d_{ij}\}$. By Corollary \ref{xxcor5.3}, the center of 
$T_{\bf q}[x_1,x_2,x_3]$ is a polynomial ring.

Conversely, suppose that $p$ is a prime such that $d_{ij}, d_{ik} \in 
p \mathbb{Z}$ for some $i, j, k$. Then $L_1, L_2, L_3 \in p \mathbb{Z}$.
This implies that $M_p$ has dimension 3. Hence, by Corollary \ref{xxcor5.2},
$\overline{Y}_p$ is not an isomorphism.
So $\overline{Y}$ is not an isomorphism. Therefore 
the center of $T_{\bf q}[x_1,x_2,x_3]$ is not a polynomial ring by 
Lemma \ref{xxlem4.3}.
\end{proof}

\begin{corollary}
\label{xxcor5.5}
The center of $T_{\bf q}[x_1,x_2,x_3,x_4]$ 
is a polynomial ring if and only if, for each prime $p$, one 
of the following holds:
\begin{enumerate}
\item 
Each $L_i \notin p \mathbb{Z}$.
\item 
For some positive integer $m$, $\overline{Y}_{p,m}$ is $4 \times 4$ with 
nonzero determinant.
\item 
There are distinct indices $i, j, k, \ell \in \{1, 2, 3, 4\}$ and a 
nonnegative integer $m$ such that $d_{ij} \in p^{m+1} \mathbb{Z}$, 
$d_{k\ell} \in p^m \mathbb{Z} - p^{m+1} \mathbb{Z}$, and every other 
$d$ term is not in $p^{m+1} \mathbb{Z}$. 
\end{enumerate}
\end{corollary}

\begin{proof} Let $P=T_{\bf q}[x_1,x_2,x_3,x_4]$.
By Lemmas \ref{xxlem4.3} and \ref{xxlem4.4}, $Z(P)$
is a polynomial ring if and only if $\overline{Y}_p$ is an isomorphism
for all $p$. It remains to show that, for each $p$, $\overline{Y}_p$
is an isomorphism if and only if one of (a), (b), or (c) holds. 
Now we fix $p$, and prove the assertion in three cases 
according to the shape of $M_p$. 

First we prove the ``if'' part. 

(a) If each $L_i \notin p \mathbb{Z}$, then $M_p = 0$ and $\overline{Y}_p$ is 
trivially an isomorphism. This handles the case when $M_p=0$.

(b) If for some $m > 0$, $\overline{Y}_{p,m}$ is $4 \times 4$ with nonzero 
determinant, then every other $\overline{Y}_{p,r}$ (for all $r\neq m$) is 
a $0 \times 0$ matrix, and consequently, $\overline{Y}_p$ an isomorphism. 
This is the case when $M_p=M_{p,m}$ is 4-dimensional for one $m$. 

(c) Assume the hypotheses in part (c). Let $m' > m$ be the integer such 
that $d_{ij} \in p^{m'} \mathbb{Z} - p^{m'+1} \mathbb{Z}$. If $m = 0$, 
then $d_{ij}$ is the only $d$ term divisible by $p$. Hence 
$\overline{Y}_{p,m'}$ is a skew-symmetric 
$2\times 2$ nonzero matrix and $\overline{Y}_{p,r}$ is trivial for all 
$r\neq m'$. Therefore  $\overline{Y}_p$ is an isomorphism. If $m > 0$, then 
$\overline{Y}_{p,m}$ and $\overline{Y}_{p,m'}$ are both skew-symmetric and 
$2 \times 2$, and (because $k_{k \ell} L_k / d_{k \ell} \notin p\mathbb{Z}$), 
nonzero. Furthermore, every other $\overline{Y}_{p,r}$ is $0 \times 0$ for all 
$r\neq m,m'$. Therefore $\overline{Y}_p$ is an isomorphism.

For the rest we prove the ``only if'' part.

Suppose that  $\overline{Y}_p$ is an isomorphism. By 
Corollary \ref{xxcor5.2}, $M_p$ is even dimensional, that is,
$\dim M_p = 0, 2$ or $4$.

The $\dim M_p = 0$ case coincides with the case when each 
$L_i \notin p \mathbb{Z}$, so we obtain case (a).

For the $\dim M_p = 2$ case, at least one $d_{ij} \in 
p \mathbb{Z}$, $L_i, L_j \in p \mathbb{Z}$, and no other $d$ term 
is a multiple of $p$, so $\overline{Y}_p$ is necessarily an isomorphism. 
We can set $m = 0$, so that $d_{ij} \in p^{m+1} \mathbb{Z}$, and all 
other $d_{ab} \notin p^{m+1} \mathbb{Z}$. So we obtain (c). 

All that remains is the $\dim M_p = 4$ case. We have that each $M_{p,m}$ 
is even dimensional by Corollary \ref{xxcor5.1}. If $\dim M_{p,m} = 4$ 
for some $m$, then $\overline{Y}_{p,m}$ is $4 \times 4$ and $\overline{Y}_p$ is an 
isomorphism if and only if $\det (\overline{Y}_{p,m}) \ne 0$. So we obtain case
(b).

Finally, suppose there exist $m' > m > 0$ such that $\dim M_{p,m} = 
\dim M_{p,m'} = 2$. Let $i, j, k, \ell$ be distinct such that 
$L_i, L_j \in p^{m'} \mathbb{Z} - p^{m'+1} \mathbb{Z}$ and 
$L_k, L_\ell \in p^m \mathbb{Z} - p^{m+1} \mathbb{Z}$. We must have that 
$d_{ij} \in p^{m'} \mathbb{Z} \subseteq p^{m+1} \mathbb{Z}$ and every 
other $d$ term is not in $p^{m+1} \mathbb{Z}$. If 
$d_{k\ell} \notin p^m \mathbb{Z}$, then 
$k_{k\ell} L_k/d_{k\ell}, k_{\ell k} L_\ell / d_{\ell k} \in p \mathbb{Z}$ 
and $\overline{Y}_{p,m}$ is the $2 \times 2$ zero matrix, yielding a 
contradiction. Therefore, $d_{k\ell}$ must be in $p^m \mathbb{Z}$. So 
we obtain case (c) again.
\end{proof}

\section{Center of Generalized Weyl algebras}
\label{xxsec6}

Let $T$ be a commutative $k$-domain. In this section we assume that
${\bf q}:=\{q_{ij}\}$ is a set of roots of unity in $T$ and 
${\mathcal A}:=\{a_{ij}\mid 1\leq i<j\leq j\}$ be a subset of $T$. Define 
the generalized Weyl algebra associated to $({\bf q},{\mathcal A})$ to be the
central $T$-algebra 
\begin{eqnarray*}
V({\bf q},{\mathcal A}): & = & 
\frac{T \langle x_{1} , \ldots ,x_{n} \rangle}{( x_{j} x_{i} -q_{ij} 
x_{i} x_{j} -a_{i j} \mid i \neq j )}.
\end{eqnarray*}
Consider a filtration on $V({\bf q},{\mathcal A})$ with $\deg x_i=1$
and $\det t=0$ for all $t\in T$. Suppose that 
\begin{equation}
\label{E6.0.1}\tag{E6.0.1}
{\text{$\gr V({\bf q},{\mathcal A})$ is naturally isomorphic to
$T_{\bf q}[x_1,\cdots,x_n]$.}}
\end{equation} 
Consider the hypothesis that
\begin{equation}
\label{E6.0.2}\tag{E6.0.2}
{\text{for any pair $(i,j)$, $a_{ij}=0$ whenever $q_{ij}=1$.}}
\end{equation}

\begin{proposition}
\label{xxpro6.1}
Suppose \eqref{E6.0.1} and \eqref{E6.0.2} and let $A=V({\bf q},{\mathcal A})$.
If the center $Z (\gr A )$ is a polynomial ring, then so 
is $Z ( A )$, and $Z(A)\cong Z (\gr A)$.
\end{proposition}

\begin{proof} If $Z (\gr A )$ is a polynomial ring, then 
$Z(\gr A)=T[x_1^{L_1},\cdots,x_n^{L_n}]$ where 
$L_i={\text{lcm}}\{ d_{ij}\mid j=1,\cdots,n\}$ [Lemma \ref{xxlem4.3}].
Recall that $d_{ij}$ is the order of $q_{ij}$. 

First we claim that $x_i^{L_i}$ is in the center of $A$. For each $j$,
we have $x_j x_i=q_{ij} x_i x_j+a_{ij}$. If $q_{ij}=1$, then $x_j$ commutes
with $x_i$ by hypothesis \eqref{E6.0.2}, so $x_j$ commutes with $x_i^{L_i}$. 
If $q_{ij} \neq 1$, then the order of $q_{ij}$ is $d_{ij}$. The equation
$x_jx_i=q_{ij}x_ix_j+a_{ij}$ implies that $x_j$ commutes with $x_i^{d_{ij}}$, as each $x_j x_i^k = q_{ij}^k x_i^k x_j + (1 + q_{ij} + \cdots + q_{ij}^{k-1}) a_{ij}$.
Since $d_{ij}$ divides $L_i$, $x_j$ commutes with $x_i^{L_i}$ for all $j\neq i$.
This shows that $x_i^{L_i}$ is central. 

Since $\gr A$ is the skew polynomial ring $T_{\bf q}[x_1,\cdots,x_n]$, 
it is easy to check that $\gr Z(A)\subset Z(\gr A)$. Since $Z(\gr A)$ 
is generated by $\{x_i^{L_i}\}_{i=1}^n$, then induction on 
the degree of element $f\in Z(A)$ shows that $f$ is generated by $x_i^{L_i}$.
Therefore the assertion follows.
\end{proof}

\begin{proposition}
\label{xxpro6.2}
Retain the above notation and suppose \eqref{E6.0.1}. 
If $a_{i j} \neq 0$ for some $i\neq j$, then $q_{i k} q_{j k} =1$ for all 
$k \neq i$ or $j$. 
\end{proposition}

\begin{proof}
We resolve $x_{k} x_{j} x_{i}$ in two different ways,
\begin{eqnarray*}
    ( x_{k} x_{j} ) x_{i} & = & ( q_{j k} x_{j} x_{k} +a_{j k} ) x_{i}\\
    & = & q_{j k} x_{j} ( x_{k} x_{i} ) +a_{j k} x_{i}\\
    & = & q_{j k} x_{j} ( q_{i k} x_{i} x_{k} +a_{i k} ) +a_{j k} x_{i}\\
    & = & q_{j k} q_{i k} ( x_{j} x_{i} ) x_{k} +q_{j k} a_{i k} x_{j} +a_{j
    k} x_{i}\\
    & = & q_{j k} q_{i k} ( q_{i j} x_{i} x_{j} +a_{i j} ) x_{k} +q_{j k}
    a_{i k} x_{j} +a_{j k} x_{i}\\
    & = & q_{j k} q_{i k} q_{i j} x_{i} x_{j} x_{k} +q_{j k} q_{i k} a_{i j}
    x_{k} +q_{j k} a_{i k} x_{j} +a_{j k} x_{i}
\end{eqnarray*}
and similarly,
\begin{eqnarray*}
    x_{k} ( x_{j} x_{i} ) & = & x_{k} ( q_{i j} x_{i} x_{j} +a_{i j} )\\
    & = & q_{i j} ( x_{k} x_{i} ) x_{j} +a_{i j} x_{k}\\
    & = & q_{i j} ( q_{i k} x_{i} x_{k} +a_{i k} ) x_{j} +a_{i j} x_{k}\\
    & = & q_{i j} q_{i k} x_{i} ( x_{k} x_{j} ) +q_{i j} a_{i k} x_{j} +a_{i
    j} x_{k}\\
    & = & q_{i j} q_{i k} q_{j k} x_{i} x_{j} x_{k} +q_{i j} q_{i k} a_{j k}
    x_{i} +q_{i j} a_{i k} x_{j} +a_{i j} x_{k}
  \end{eqnarray*}
Comparing the coefficients of $x_{k}$ gives the result.
\end{proof}

When an algebra $A$ is finitely generated and free over its center
(as in the situation of Proposition \ref{xxpro6.1}), one should be able 
to compute the discriminant of $A$ over its center. We give an example 
here. 

\begin{example}
\label{xxex6.3} Let $A$ be generated by $x_1,x_2,x_3,x_4$ subject to the 
relations
\begin{align}\notag
x_3 x_1-x_1x_2 =0, &\quad x_4x_2+x_2x_4=0,\\
\label{E6.3.1}\tag{E6.3.1}
x_3 x_2-x_2x_3 =0, &\quad x_3x_4+x_4x_3=0,\\
\notag
x_4 x_1+x_1x_4 =0, &\quad x_1x_2+x_2x_1=x_3^2+x_4^2.
\end{align}
This is the example in \cite[Lemma 1.1]{VVR} (with $\lambda=0$). 
It is an iterated Ore extension, and therefore, Artin-Schelter regular
of global dimension four. 

It is not hard to check that the center of $A$ is generated by $x_i^2$. 
This algebra is a factor ring of the algebra $B$ over $T:=k[t]$ generated by 
$x_1,x_2,x_3,x_4$ subject to the 
relations
\begin{align}
\notag
x_3 x_1-x_1x_2 =0, &\quad x_4x_2+x_2x_4=0,\\
\label{E6.3.2}\tag{E6.3.2}
x_3 x_2-x_2x_3 =0, &\quad x_3x_4+x_4x_3=0,\\
\notag
x_4 x_1+x_1x_4 =0, &\quad x_1x_2+x_2x_1=t.
\end{align}

Note that $\gr B$ is a skew polynomial ring over $T$ with the above 
relations by setting $t=0$. The $Y$-matrix is
$$\begin{pmatrix} 0& 1/2 &0 & 1/2\\
-1/2& 0& 0& 1/2\\
0&0&0& 1/2\\
-1/2&-1/2&-1/2&0\end{pmatrix}.$$
By Corollary \ref{xxcor5.5}(b), $B$ has center 
$T[x_1^2,x_2^2,x_3^2,x_4^2]$. By the next lemma the discriminant
of $B$ over its center is  $2^{48} (4 x_1^2 x_2^2 -t^2)^8
x_3^{16}x_4^{16}$. By Lemma \ref{xxlem1.2},
the discriminant of $A$ over its center is  $2^{48} 
(4 x_1^2 x_2^2 -(x_3^2+x_4^2)^2)^8 x_3^{16}x_4^{16}$.
We will see in the next sections that 
${\mathbb D}(A)=A$. As a consequence of 
Theorem \ref{xxthm0.5}, $A$ is cancellative and 
the automorphism group of $A$ is affine.
\end{example}

\begin{lemma}
\label{xxlem6.4} Suppose the $k[t]$-algebra $B$ is generated by 
$\{x_1,x_2,x_3,x_4\}$ subject to the six relations given \eqref{E6.3.2}.
Then the discriminant of $B$ over its center is $2^{48} (4 x_1^2 x_2^2 -t^2)^8
x_3^{16}x_4^{16}$.
\end{lemma}

\begin{proof}[Sketch of the Proof]
It is routine to check that the center of $B$ is 
$$Z(B)=k[t][x_1^2, x_2^2, x_3^2, x_4^2].$$ 
The algebra $B$ is a free module over $Z(B)$ of rank 16 with a 
$Z(B)$-basis $\{x_1^a x_2^b x_3^c x_4^d\mid a,b,c,d =0,1\}$. 
Let $\{z_1,\cdots z_{16}\}$ be the above $Z(B)$-basis. Then 
we can compute the matrix $(\mathrm{tr}(z_i z_j))_{16\times 16}$,
which is

{\tiny
\begin{eqnarray*}
\left( \begin{array}{rrrrrrrrrrrrrrrr}
    16 & 0 & 0 & 0 & 0 & 8 t & 0 & 0 & 0 & 0 & 0 & 0 & 0 & 0 & 0 & 0\\
    0 & 16 a & 8 t & 0 & 0 & 0 & 0 & 0 & 0 & 0 & 0 & 0 & 0 & 0 & 0 & 0\\
    0 & 8 t & 16 b & 0 & 0 & 0 & 0 & 0 & 0 & 0 & 0 & 0 & 0 & 0 & 0 & 0\\
    0 & 0 & 0 & 16 c & 0 & 0 & 0 & 0 & 0 & 0 & 0 & 8 c t & 0 & 0 & 0 & 0\\
    0 & 0 & 0 & 0 & 16 d & 0 & 0 & 0 & 0 & 0 & 0 & 0 & 8 d t & 0 & 0 & 0\\
    8 t & 0 & 0 & 0 & 0 & \alpha & 0 & 0 & 0 & 0 & 0 & 0 & 0 & 0 &
    0 & 0\\
    0 & 0 & 0 & 0 & 0 & 0 & 16 a c & 0 & 8 c t & 0 & 0 & 0 & 0 & 0 & 0 & 0\\
    0 & 0 & 0 & 0 & 0 & 0 & 0 & -16 a d & 0 & -8 d t & 0 & 0 & 0 & 0 & 0 & 0\\
    0 & 0 & 0 & 0 & 0 & 0 & 8 c t & 0 & 16 b c & 0 & 0 & 0 & 0 & 0 & 0 & 0\\
    0 & 0 & 0 & 0 & 0 & 0 & 0 & -8 d t & 0 & -16 b d & 0 & 0 & 0 & 0 & 0 & 0\\
    0 & 0 & 0 & 0 & 0 & 0 & 0 & 0 & 0 & 0 & -16 c d & 0 & 0 & 0 & 0 & -8 c d
    t\\
    0 & 0 & 0 & 8 c t & 0 & 0 & 0 & 0 & 0 & 0 & 0 & \beta & 0
    & 0 & 0 & 0\\
    0 & 0 & 0 & 0 & 8 d t & 0 & 0 & 0 & 0 & 0 & 0 & 0 &  \gamma
    & 0 & 0 & 0\\
    0 & 0 & 0 & 0 & 0 & 0 & 0 & 0 & 0 & 0 & 0 & 0 & 0 & 16 a c d & 8 c d t &
    0\\
    0 & 0 & 0 & 0 & 0 & 0 & 0 & 0 & 0 & 0 & 0 & 0 & 0 & 8 c d t & 16 b c d &
    0\\
    0 & 0 & 0 & 0 & 0 & 0 & 0 & 0 & 0 & 0 & -8 c d t & 0 & 0 & 0 & 0 & \delta
  \end{array} \right)
\end{eqnarray*}
}
where $\alpha=-16 a b + 8 t^{2},
\beta=-16 a b c + 8 c t^{2},
\gamma=-16 a b d + 8 d t^{2},
\delta=16 a bc d - 8 c d t^{2}$, and 
$a=x_1^2, b=x_2^2, c=x_3^2, d=x_4^2$. We skip the details in  
computing the above traces. By using Maple, its determinant is
$2^{48} (4 x_1^2 x_2^2 -t^2)^8 x_3^{16}x_4^{16}$.
\end{proof}

\section{Three subalgebras}
\label{xxsec7}

In this section we discuss three (possibly different) subalgebras
of $A$, all of which are helpful for the applications in the next
section. 

\subsection{Makar-Limanov Invariants}
\label{xxsec7.1}

The first subalgebra is the Makar-Limanov Invariants of $A$ 
introduced by Makar-Limanov \cite{Ma1}. This invariant has been very 
useful in commutative algebra.
For any $k$-algebra $A$, let ${\rm Der}(A)$ denote the 
set of all $k$-derivations of $A$ and $\LND(A)$ denote the 
set of locally nilpotent $k$-derivations of $A$. 

\begin{definition}
\label{xxdef7.1}
Let $A$ be an algebra over $k$.
\begin{enumerate}
\item[(1)]
The {\it Makar-Limanov invariant} \cite{Ma1} of $A$ is defined to be 
\begin{equation}
\label{E7.1.1}\tag{E7.1.1}
{\rm ML}(A) \ = \ \bigcap_{\delta\in {\rm LND}(A)} {\rm ker}(\delta).
\end{equation}
\item[(2)]
We say that $A$ is \emph{$\LND$-rigid} if ${\rm ML}(A)=A$,
or $\LND(A)=\{0\}$. 
\item[(3)]
We say that $A$ is \emph{strongly $\LND$-rigid} if ${\rm ML}
(A[t_1,\cdots,t_d])=A$ for all $d\geq 0$.
\end{enumerate}
\end{definition}
The following lemma is clear. Part (2) follows
from the fact that $\partial\in \LND(A)$ if and only if 
$g^{-1} \partial g\in \LND(A)$.

\begin{lemma}
\label{xxlem7.2}
Let $A$ be an algebra.
\begin{enumerate}
\item[(1)]
${\rm ML}(A)$ is a subalgebra of $A$.
\item[(2)]
For any $g\in \Aut(A)$, $g({\rm ML}(A))={\rm ML}(A)$.
\end{enumerate}
\end{lemma}

\subsection{Divisor subalgebras}
\label{xxsec7.2}
Throughout this subsection let $A$ be a domain containing 
${\mathbb Z}$. Let $F$ be a subset of $A$. Let $Sw(F)$ be the set 
of $g\in A$ such that $f=agb$ for some $a,b\in A$ and $0\neq f\in F$. 
Here $Sw$ stands for ``sub-word'', which can be viewed as a divisor.

\begin{definition}
\label{xxdef7.3} Let $F$ a subset of $A$.
\begin{enumerate}
\item[(1)]
Let $D_{0} ( F ) =F$. Inductively define $D_{n} ( F )$ as the $k$-subalgebra
of $A$ generated by $S w ( D_{n-1} ( F ) )$. The subalgebra $D ( F ) =
\bigcup_{n \geq 0} D_{n} ( F )$ is called the 
{\it $F$-divisor subalgebra of $A$}. If $F$ is the singleton $\{f\}$, 
we simply write $D(\{f\})$ as $D(f)$
\item[(2)]
If $f=d(A/Z)$ (if it exists), we call $D(f)$ {\it the 
discriminant-divisor subalgebra of $A$} or {\it DDS of $A$},
and write it as ${\mathbb D}(A)$.
\end{enumerate}
\end{definition}

The following lemma is well-known \cite[p. 4]{Ma2}.

\begin{lemma}
\label{xxlem7.4}
Let $x,y$ be nonzero elements in $A$ and let
$\partial\in \LND(A)$. If $\partial(xy)=0$, then 
$\partial(x)=\partial(y)=0$.
\end{lemma}

\begin{proof} Let $m$ and $n$ be the largest integers such that
$\partial^m(x)\neq 0$ and $\partial^n(y)\neq 0$. Then
the product rule and the choice of $m,n$ imply that
$$\partial^{m+n}(xy)=\sum_{i=0}^{m+n} {n+m\choose i}
\partial^i(x)\partial^{m+n-i}(y)={n+m\choose m}
\partial^m(x)\partial^n(y)\neq 0.$$ 
So $m+n=0$. The assertion follows.
\end{proof}

\begin{lemma}
\label{xxlem7.5}
Let $F$ be a subset of ${\rm ML}(A)$. Then $D(F)\subseteq {\rm ML}(A)$.
\end{lemma}

\begin{proof} Let $\partial$ be any element in $\LND(A)$.
By hypothesis, $\partial(f)=0$ for all $f\in F$.
By Lemma \ref{xxlem7.4}, $\partial(x)=0$ for all $x\in Sw(F)$. 
So $\partial=0$ when restricted to $D_1(F)$. By 
induction, $\partial=0$ when restricted to $D(F)$. The assertion
follows by taking arbitrary $\partial\in \LND(A)$.
\end{proof}

\begin{lemma}
\label{xxlem7.6} Suppose $d(A/Z)$ is defined.
Then the DDS ${\mathbb D}(A)$ is preserved by all
$g\in \Aut(A)$.
\end{lemma}

\begin{proof} By \cite[Lemma 1.8(6)]{CPWZ1} or 
\cite[Lemma 1.4(4)]{CPWZ2}, $d(A/Z)$ is $g$-invariant 
up to a unit. So, if $g\in \Aut(A)$, then $g$ 
maps $Sw(d(A/Z))$ to $Sw(d(A/Z))$
and $D_1(d(A/Z))$ to $D_1(d(A/Z))$. By induction, one sees
that $g$ maps $D_n(d(A/Z))$ to $D_n(d(A/Z))$. So the 
assertion follows.
\end{proof}

We need to find some elements $f\in A$ so that $\partial(f)=0$
for all $\partial\in \LND(A)$. The next lemma was proven in 
\cite[Proposition 1.5]{CPWZ2}.

\begin{lemma}
\label{xxlem7.7}
Let $Z$ be the center of $A$ and $d\geq 0$. Suppose $A^\times= k^{\times}$. 
Assume that $A$ is finitely generated and free over $Z$. Then 
$\partial(d(A/Z))=0$ for all $\partial\in \LND(A[t_1,\cdots,t_d])$.
\end{lemma}

\begin{proof} Let $f$ denote the element 
$d(A[t_1,\cdots,t_d]/Z[t_1,\cdots,t_d])$ in $Z[t_1,\cdots,t_d]$.
By \cite[Proposition 1.5]{CPWZ2}, $\partial (f)=0$.
By \cite[Lemma 5.4]{CPWZ1}, 
$$f
=_{k^\times}d(A/Z).$$ 
The assertion follows.
\end{proof}

Here is the first relationship between the two subalgebras.

\begin{proposition}
\label{xxpro7.8} Retain the hypothesis of Lemma \ref{xxlem7.7}.
Let $d\geq 0$. Then 
$${\mathbb D}(A)\subseteq {\rm ML}(A[t_1,\cdots,t_d])\subseteq A.$$
\end{proposition}

\begin{proof} It is clear that ${\rm ML}(A[t_1,\cdots,t_d])\subseteq A$
by \cite{BZ}. Let $f=d(A/Z)$, which is in $A\subseteq A[t_1,\cdots,t_d]$.
By Lemma \ref{xxlem7.7}, $f\in {\rm ML}(A[t_1,\cdots,t_d])$. Let
$D'(f)$ be the discriminant-divisor subalgebra of $f$ in
$A[t_1,\cdots,t_d]$.
By Lemma \ref{xxlem7.5}, $D'(f)\subseteq {\rm ML}(A[t_1,\cdots,t_d])$.
It is clear from the definition that $D(f)\subseteq D'(f)$. 
Therefore the assertion follows.
\end{proof}

In particular, by taking $d=0$, we have 
${\mathbb D}(A)\subseteq {\rm ML}(A)$.

\subsection{Aut-Bounded subalgebra}
\label{xxsec7.3}
In this subsection we assume that $A$ is filtered such that the 
associated graded ring $\gr A$ is a connected graded domain.
Later we further assume that $A$ is connected graded. Since
$\gr A$ is a connected graded domain, we can define 
$\deg f$ to be the degree of $\gr f$, and the degree satisfies
the equation
$$\deg (xy)=\deg x+\deg y$$
for all $x,y \in A$.

\begin{definition}
\label{xxdef7.9}
Retain the above hypotheses. Let $G$ be a subgroup of $\Aut(A)$
and let $V$ be a subset of $A$. 
\begin{enumerate}
\item[(1)]
Let $x$ be an element in $A$. The {\it $G$-bound} of $x$ 
is defined to be
$$\deg_G(x) :=\sup \{\deg (g(x))\mid g\in G\}.$$
\item[(2)] 
Let $g$ be in $\Aut(A)$.
The {\it $V$-bound} of $g$ is defined to be
$$\deg_g(V) :=\sup \{\deg (g(x))\mid x\in V\}.$$
\item[(3)]
The {\it $G$-bounded subalgebra} of $A$, denoted by
$\beta_{G}(A)$, is the set of elements $x$ in
$A$ with finite $G$-bound. It is clear that $\beta_G(A)$ is 
a subalgebra of $A$ [Lemma \ref{xxlem7.10}(1)]. In particular,
the {\it $\Aut$-bounded subalgebra} of $A$, denoted by
$\beta(A)$, is the set of elements $x$ in
$A$ with finite $\Aut(A)$-bound.
\end{enumerate}
\end{definition}

The following lemma is easy, so we omit the proof. 

\begin{lemma}\label{xxlem7.10}
Retain the above notation. Let $G$ be a subgroup
of $\Aut(A)$. 
\begin{enumerate}
\item[(1)]
The set $\beta_{G}(A)$ is a subalgebra of $A$.
\item[(2)]
$g(\beta_{G}(A))=\beta_{G}(A)$ for all $g\in G$.
\end{enumerate}
\end{lemma}

Here is the relation between the two subalgebras  ${\mathbb D}(A)$
and $\betaAut(A)$. Let $V$ be a subset of $A$. We say $V$ is of
{\it bounded degree} if there is an $N$ such that $\deg(v)<N$ for all $v\in V$.

\begin{proposition}
\label{xxpro7.11}
Let $A$ be a filtered algebra such that $\gr  A$ is a connected graded
domain. Suppose that $G \subseteq \Aut ( A )$ and $F \subseteq A$.
\begin{enumerate}
\item[(1)] 
If $G(F)$ has bounded degree,
then $D(F) \subseteq \beta_{G}(A)$.
\item[(2)] 
If $f \in A$ is such that $g (f)=_{Z ( A )^{\times}} f$ for all $g \in G$,
then $D(f) \subseteq \beta_{G} ( A )$.
\item[(3)] 
Assume that $A$ is finitely generated and free over 
its center $Z$. Let $f=d(A/Z)$, then ${\mathbb D}(A)=D(f) \subseteq \beta ( A )$.
\end{enumerate} 
\end{proposition}

\begin{proof}
(1) We have $D_{0}( F ) = F \subseteq \beta_{G} ( A )$ by assumption
and use induction on $n$. 
Suppose that $D_{n-1} ( F ) \subseteq \beta_{G} ( A )$. Assume that 
$D_{n} ( F )$ is not contained in $\beta_{G} ( A )$. Then there exists 
$x \in D_{n} ( A )$ such that $G(x)$ does not have bounded degree.
Since $D_{n} ( A )$ is generated by $Sw(D_{n-1}(A))$ as an algebra, 
there is an $f\in Sw(D_{n-1}(A))$ such that $G(f)$ does not have bounded
degree. By definition of $Sw(D_{n-1}(A))$, there exists a nonzero 
$f' \in D_{n-1} ( A )$ and $a,b \in A$ such that $f' =a f b$.
Since $\gr  A$ is a domain, we have $\deg ( g (f') ) = \deg ( g (a) ) 
+ \deg( g (f) ) + \deg ( g (b) )$ for all $g \in G$. Hence $G(f')$ 
does not have bounded degree, which is a contradiction. Hence 
$D_{n} ( F ) \subseteq \beta_{G} ( A)$ for all $n \geq 1$, 
therefore $D ( F ) \subseteq \beta_{G} ( A )$.

(2) Since $Z ( A )^{\times} \subseteq A_{0}$, we see that $G(f)$ has 
bounded degree, hence part (2) follows from part (1).

(3) The third assertion is a special case of part (2) by 
Lemma \ref{xxlem1.2}.
\end{proof}

Under the hypotheses of Propositions \ref{xxpro7.8} and \ref{xxpro7.11}
(and assume that $A$ is finitely generated and free over its center $Z$),
we have
\begin{center}
\begin{tikzpicture}
\node (a) at (0,0) {$\mathrm{ML}(A)$};
\node (b) at (1,1) {$\mathbb{D}(A)$};
\node (c) at (2,0) {$\beta(A)$};
\node (d) at (1,-1) {$A$};
\node at ($(a)!0.5!(b)$) {\rotsymp{45}};
\node at ($(b)!0.5!(c)$) {\rotsym{315}};
\node at ($(c)!0.5!(d)$) {\rotsymp{45}};
\node at ($(d)!0.5!(a)$) {\rotsym{315}};
\end{tikzpicture}
\end{center}
For the rest of this section, we assume that $A$ is a connected graded
domain and that $k$ contains the field ${\mathbb Q}$. 
An automorphism $g$ of $A$ is called {\it unipotent} if 
\begin{equation}
\label{E7.11.1}\tag{E7.11.1}
g(v) = v + {\text{(higher degree terms)}}
\end{equation}
for all homogeneous elements $v\in A$. Let $\Aut_{uni}(A)$ denote
the subgroup of $\Aut(A)$ consisting of unipotent automorphisms
\cite[After Theorem 3.1]{CPWZ2}. If $g\in \Aut_{uni}(A)$, 
we can define 
\begin{equation}
\label{E7.11.2}\tag{E7.11.2}
\log (g):= - \sum_{i=1}^{\infty} \frac{1}{i} (1-g)^i.
\end{equation}

Let $C$ be the completion of $A$ with respect to the graded maximal 
ideal ${\mathfrak m}:=A_{\geq 1}$. Then $C$ is a local ring 
containing $A$ as a subalgebra. We can define $\deg_l:C\to \mathbb{Z}$ 
by setting $\deg_l(v)$ to be the lowest degree of the nonzero homogeneous
components of $v\in C$. We define {\it a unipotent 
automorphism of $C$} in a similar way to \eqref{E7.11.1} by using $\deg_l$. 
It is clear that if $g\in \Aut_{uni}(A)$, then it induces a unipotent 
automorphism of $C$, which is still denoted by $g$.

\begin{lemma}
\label{xxlem7.12} 
Let $A$ be a connected graded domain. Let $g\in \Aut_{uni}(A)$ and 
$G$ be any subgroup of $\Aut(A)$ containing $g$. Let $B$ denote 
$\beta_G(A)$. Then $\log(g)\mid_{B}$ is a locally nilpotent derivation 
of $B$. Further, $g\mid_{B}$ is the identity if and only if 
$\log(g)\mid_{B}$ is zero.
\end{lemma}

\begin{proof} Let $C$ be the completion of $A$ with respect 
to the graded maximal ideal ${\mathfrak m}:=A_{\geq 1}$. 
Let $g$ also denote the algebra automorphism of $C$ induced by $g$.
Then $g$ is also a unipotent automorphism of $C$.

Since $g$ is unipotent, $\deg_l (1-g)(v)> \deg_l v$ for any 
$0\neq v\in C$. By induction, one has
$\deg (1-g)^n (v)\geq n+\deg v$ for all $n\geq 1$. 
Thus $\log(g)(v)$ converges and therefore is well-defined. 
It follows from a standard argument that $\log(g)$ is a derivation of 
$C$ (this also is a consequence of \cite[Proposition 2.17(b)]{Fr}).

Let $v$ be an element in $B:=\beta_{G}(A)$. Note that $g^n(v)\in B$
for all $n$ by Lemma \ref{xxlem7.10}. Since $v\in 
B$, there is an $N_0$ such that  $\deg g^n(v)<N_0$
for all $n$. If $(1-g)^n(v)\neq 0$, then 
\begin{equation}
\label{E7.12.1}\tag{E7.12.1}
\deg (1-g)^n(v)=\deg (\sum_{i=0}^n {n \choose i} g^i(v))< N_0, 
\quad {\text{for all $n$}}.
\end{equation}
When $n\geq N_0$, the inequalities from the previous paragraph
imply that 
\begin{equation}
\label{E7.12.2}\tag{E7.12.2}
\deg_l (1-g)^n (v)\geq n+\deg v\geq N_0,
\end{equation}
which contradicts \eqref{E7.12.1} unless $(1-g)^n (v)=0$. Therefore
\begin{equation}
\label{E7.12.3}\tag{E7.12.3}
(1-g)^n (v)=0, \quad {\text{for all $n>N_0$}}.
\end{equation}

By \eqref{E7.12.3}, the infinite sum of $\log(g)$ in \eqref{E7.11.2} 
terminates when applied to $v\in B$, and $\log(g)(v)\in A$.
By Lemma \ref{xxlem7.10}, $\log(g)(v)\in B$. Since
$\log(g)$ is a derivation of $C$, it is a derivation when restricted 
to $B$. 

Next we need to show that it is a locally nilpotent derivation 
when restricted to $B$. It suffices to verify that, for any $v\in B$, 
$\log(g)^N(v)=0$ for $N\gg 0$, which follows from \eqref{E7.11.2}
and \eqref{E7.12.3}.

The final assertion follows from the fact that $g$ is the 
exponential function of $\log(g)$ and $\log(g)$ is locally nilpotent.
\end{proof}

Now we are ready to prove the second part of Theorem \ref{xxthm0.5}
without the finite GK-dimension hypothesis.

\begin{theorem}
\label{xxthm7.13}
Let $k$ be a field of characteristic zero and 
$A$ be a connected graded domain over $k$. Assume that $A$ is finitely generated and 
free over its center $Z$ in part {\rm{(2)}}.
\begin{enumerate}
\item[(1)]
If ${\rm ML}(A)=\beta(A)=A$, then $\Aut_{uni}(A)=\{1\}$.
\item[(2)]
If ${\mathbb D}(A)=A$, then $\Aut_{uni}(A)=\{1\}$.
\end{enumerate}
\end{theorem} 

\begin{proof} (1) By hypothesis, $B:=\beta(A)$ equals $A$. 
Let $g\in \Aut_{uni}(A)$. Then $\log(g)\mid_B$ is a locally 
nilpotent derivation of $B$ by Lemma \ref{xxlem7.12}. Hence
$\log(g)\in \LND(A)$. Since ${\rm ML}(A)=A$, $\LND(A)=\{0\}$.
So $\log(g)=0$. By Lemma \ref{xxlem7.12}, $g$ is the identity.

(2) Combining the hypothesis ${\mathbb D}(A)=A$ with 
Propositions \ref{xxpro7.8} and \ref{xxpro7.11}, we have
${\rm ML}(A)=\beta(A)=A$. The assertion follows from
part (1).
\end{proof}

\section{Applications}
\label{xxsec8}

In this section we assume that $k$ is a field of characteristic zero.

\subsection{Zariski cancellation problem}
\label{xxsec8.1}
The Zariski cancellation problem for noncommutative algebras was studied 
in \cite{BZ}. We recall some definitions and results.

\begin{definition}\cite[Definition 1.1]{BZ}
\label{xxdef8.1}
Let $A$ be an algebra.
\begin{enumerate}
\item
We call $A$ {\it cancellative} if $A[t]\cong B[t]$ for
some algebra $B$ implies that $A\cong B$.
\item
We call $A$ {\it strongly cancellative} if, for any $d\geq 1$, 
$A[t_1,\ldots,t_d]\cong B[t_1,\ldots,t_d]$ for some algebra 
$B$ implies that $A\cong B$.
\end{enumerate}
\end{definition}

The original Zariski cancellation problem, denoted by {\bf ZCP}, asks 
if the polynomial ring $k[t_1,\cdots, t_n]$, where $k$ is a field, is 
cancellative. A recent result of Gupta \cite{Gu1,Gu2} settled the 
question {\bf ZCP} negatively in positive characteristic for $n\geq 3$. 
The {\bf ZCP} in characteristic zero remains open for $n\geq 3$. 
Some history and partial results about the {\bf ZCP} can be found 
in \cite{BZ}. In \cite{BZ}, the authors used discriminants and locally 
nilpotent derivations to study the Zariski cancellation problem for 
noncommutative rings.

One of the main results in \cite{BZ} is the following.

\begin{theorem}\cite[Theorems 3.3 and 0.4]{BZ}
\label{xxthm8.2}
Let $A$ be a finitely generated domain of finite Gelfand-Kirillov 
dimension. If $A$ is strongly $\LND$-rigid {\rm{(}}respectively, 
$\LND$-rigid{\rm{)}}, then $A$ is strongly cancellative 
{\rm{(}}respectively, cancellative{\rm{)}}. 
\end{theorem}

Now we have an immediate consequence, which is the first
part of Theorem \ref{xxthm0.5}. Combining with Theorem 
\ref{xxthm7.13}, we finished the proof of Theorem \ref{xxthm0.5}.

\begin{theorem}
\label{xxthm8.3}
Let $A$ be a finitely generated domain of finite GK-dimension. 
Let $Z$ be the center of $A$ and suppose $A^\times= k^{\times}$. 
Assume that $A$ is finitely generated and free over $Z$. If $A=
{\mathbb D}(A)$, then $A$ is strongly cancellative.
\end{theorem}

\begin{proof} Combining the hypothesis $A={\mathbb D}(A)$
with Proposition \ref{xxpro7.8}, we have
$$A={\mathbb D}(A)\subseteq {\rm ML}(A[t_1,\cdots,t_d])
\subseteq A.$$
So ${\rm ML}(A[t_1,\cdots,t_d]) = A$, or $A$ is 
strongly $\LND$-rigid. The assertion follows from 
Theorem \ref{xxthm8.2}. 
\end{proof}

Next we give two examples. 

\begin{example}
\label{xxex8.4} Let $A$ be generated by $x_1,\cdots,x_4$ subject to the
relations
\begin{align}
\notag
x_1x_2+x_2x_1=0,\;\; \quad & x_2x_3+x_3x_2=0,\\
\label{E8.4.1}\tag{E8.4.1}
x_1x_3+x_3x_1=0,\;\; \quad & x_3x_4+x_4x_3=0,\\
\notag
x_1x_4+x_4x_1=x_3^2, \quad & x_2x_4+x_4x_2=0.
\end{align}
This is an iterated Ore extension, so it is Artin-Schelter regular of global
dimension 4. 
This is a special case of the algebra in \cite[Definition 3.1]{VVRW}.
Set $x_i^2=y_i$ for $i=1,\cdots,4$. Then $Z(A)=k[y_1,y_2,y_3,y_4]$.
The $M_1$-matrix of \eqref{E3.0.1} is
$$(a_{ij})_{4\times 4}=\begin{pmatrix}
2y_1 & 0 & 0 & y_3\\
0 & 2y_2 & 0 & 0 \\
0 & 0 & 2y_3 & 0\\
y_3 & 0 & 0 & 2 y_4
\end{pmatrix}.
$$
The determinant $\det (a_{ij})$ is $f_0:=4y_2y_3(4y_1y_4-y_3^2)$ by 
linear algebra. By Theorem \ref{xxthm3.7}, the discriminant 
$f:=d(A/Z)$ is $f_0^{2^3}$. It is clear that $y_2, y_3\in Sw(f)$
and $y_1,y_4\in Sw(D_1(f))$. Thus $x_i\in Sw(D_2(f))$ for all $i$.
Consequently, $A={\mathbb D}(A)$. By Theorem \ref{xxthm8.3}, $A$ is 
strongly cancellative.
\end{example}

The next example is somewhat generic.

\begin{example}
\label{xxex8.5}
Let $T$ be a commutative domain, and $A=C(V,q)$ be the Clifford 
algebra associated to a quadratic form $q:V \rightarrow T$ 
where $V$ is a free $T$-module of rank $n$. Suppose that $n$ is
even. Then the center of $A$ is $T$ 
\cite[Chapter 5, Theorem 2.5(a)]{La}. We assume that 
$A$ is a domain with $A^{\times} =k^{\times}$. Let $t_{1},\ldots ,t_{w}$ 
be a set of generators of $T$, and suppose that $q(V)\subseteq 
( t_{1} \cdots t_{w} ) T$ or $\det(q) \in ( t_{1} \cdots t_{w} ) T$. 
Then by Theorem \ref{xxthm3.7} we have 
$f \assign d (A/T) \in ( t_{1} \cdots t_{w})^{2^{n-1}}$.
So $t_s\in Sw(f)$ for all $s$. This shows that $T\subseteq {\mathbb D}(A)$ 
and then $A={\mathbb D}(A)$ (as $x_i^2\in T$). By Theorem \ref{xxthm8.3},
$A$ is strongly cancellative.
\end{example}

\begin{remark}
\label{xxrem8.6} Let $A$ be the algebra in Example \ref{xxex6.3}.
Using the formula for $d(A/Z)$ given in Lemma \ref{xxlem6.4}, 
it is easy to see that $A={\mathbb D}(A)$. 
So $A$ is cancellative by Theorem \ref{xxthm8.3}.
\end{remark}

\subsection{Automorphism problem}
\label{xxsec8.2}
By \cite{CPWZ1, CPWZ2}, the discriminant controls the automorphism group 
of some noncommutative algebras. In this section we compute some
automorphism groups by using the discriminants computed in previous sections.
We first recall some definitions and results.

We modify the definitions in \cite{CPWZ1, CPWZ2} slightly.
Let $A$ be an ${\mathbb N}$-filtered algebra such that 
$\gr A$ is a connected graded domain. Let $X:=\{x_1,\cdots, x_n\}$
be a set of elements in $A$ such that it generates $A$ and $\gr X$ 
generates $\gr A$. We do not require $\deg x_i=1$ for all $i$.

\begin{definition}
\label{xxdef8.7}
Let $f$ be an element in $A$ and $X'=\{x_1,\cdots,x_m\}$ be a 
subset of $X$. We say $f$ is \emph{dominating} over $X'$ if for any 
subset $\{y_1,\cdots, y_n\}\subseteq A$ that is linearly
independent in the quotient $k$-space $A/k$, there is a lift
of $f$, say $F(X_1,\cdots,X_n)$, in the free algebra $k\langle
X_1,\cdots,X_n\rangle$, such that $\deg F(y_1,\cdots, y_n)>\deg f$
whenever $\deg y_i> \deg x_i$ for some $x_i\in X'$.
\end{definition}

The following lemma is easy.

\begin{lemma}
\label{xxlem8.8}
Retain the above notation.
Suppose $f: =d(A/Z)$ is dominating over $X'$. Then for every
automorphism $g\in \Aut(A)$, $\deg g(x_i)\leq \deg x_i$ for all
$x_i\in X'$.
\end{lemma}

\begin{proof}
Let $y_i=g(x_i)$. Then $\{y_1,\cdots, y_n\}$ is linearly independent in 
$A/k$ (as $\{x_1,\cdots,x_n\}$ is linearly independent on $A/k$).
If $\deg y_i>\deg x_i$ for some $i$, by the dominating property,
there is a lift of $f$ in the free algebra, say $F(X_1,\cdots,X_n)$,
such that $\deg F(y_1,\cdots,y_n)>\deg f$. Since $g$ is an algebra
automorphism, 
$$F(y_1,\cdots, y_n)=F(g(x_1),\cdots,g(x_n))=g(F(x_1,\cdots,x_n))
=g(f).$$
By \cite[Lemma 1.8(6)]{CPWZ1}, $g(f)=f$ (up to a unit in $Z$). Hence
$$\deg F(y_1,\cdots,y_n)=\deg g(f)=\deg f,$$ 
yielding a contradiction. Therefore
$\deg g(x_i)=\deg y_i\leq \deg x_i$ for all $i$.
\end{proof}

We will study the automorphism group of a class of Clifford algebras,
see Example \ref{xxex8.5}. 

\begin{example}
\label{xxex8.9}
Let $A$ be the Clifford algebra over a commutative $k$-domain
$T$ as in Example \ref{xxex8.5} and assume that $n$ is even. 
Let $\{z_1,\cdots,z_n\}$ denote a set of generators for $A$.
We will use $\{x_1,\cdots,x_n\}$ for the generators of the generic 
Clifford algebra $A_g$ defined in Section \ref{xxsec3}. Then
there is an algebra homomorphism from $A_g\to A$ sending $x_i$
to $z_i$ for all $i$. Since $n$ is even, $T$ is the center
of $A$. Assume that $A$ is a filtered algebra 
such that $\gr A$ is a connected graded domain, so we can define 
the degree of any non-zero element in $A$. Further assume that 
$\deg t_i=2$ (not $1$) for all $i=1,\cdots,w$ and $\deg z_i>2$ 
for all $i=1,2,\cdots,n$. In particular, there is no
element of degree 1. Some explicit examples are given later
in this example.

Recall that we assumed $q(V)\subseteq(t_1\cdots t_w)T$. Let 
$2b_{ij} = z_jz_i+z_iz_j$, then we can write $b_{ij} = 
(t_1\cdots t_w)^N b'_{ij}$ for some $N>0$. By Theorem \ref{xxthm3.7}, 
the discriminant is $f:=d(A/T)=[(\prod_{s=1}^w t_s)^N d']^{2^{n-1}}$ 
where $d'=\det(2b'_{ij})_{n\times n}$. We need another
extra hypothesis, which is that \begin{equation}
\label{E8.9.1}
\tag{E8.9.1}
\deg d'<N.
\end{equation} 
Let $X'=\{t_i\}_{i=1}^w$ and $X=\{z_i\}_{i=1}^n \bigcup X'$. Then $f$ is 
a noncommutative polynomial over $X'$. We first claim that $f$ is 
dominating over $X'$. Let $\{y_i\}_{i=1}^w$ be a set of elements in 
$A\setminus k$. If $\deg y_i>2$ for some $i$, then 
$\deg [(\prod_{s=1}^w y_s)^N d'(y_1,\cdots,y_w)]^{2^{n-1}}$ is
strictly larger than the degree of $f$, as we assume that 
$\deg d'<N.$ This shows the claim. 

Now let $g$ be any algebra automorphism of $A$ and let $y_i$ be
$g(t_i)$ for all $i$. Then, by Lemma \ref{xxlem8.8}, $\deg y_i=2$. 
It follows from the relations $z_i z_i=b_{ii}$ that $\deg z_i>3$. Hence 
$(\gr A)_2$ is generated by the $t_i$'s.
This implies that $y_i$ is in the span of $X'$ and $k$. In some 
sense, every automorphism of $A$ is affine (with respect to $X'$). 
It is a big step in understanding the automorphism group of $A$.

Below we study the automorphism group of a family of subalgebras 
of the generic Clifford algebra $A_g$ of rank $n$ that is defined 
in Section \ref{xxsec3}. As before we assume $n$ is even. 
We have two different sets of variables $t$, one for $A_g$ and the 
other for general $A$. It would be convenient to unify these 
in the following discussion. So we identify $\{t_{(i,j)}\mid 1\leq i\leq j\leq  n\}$
with $\{t_i\}_{i=1}^w$ via a bijection $\phi$. Here $w=\frac{1}{2}
n(n+1)$ as in the definition of $A_g$ [Section \ref{xxsec3}]. 

Let $r$ be any positive integer and let $B_{g,r}$ be the graded 
subalgebra of $A_g$ generated by $\{t_{(i,j)}\}$ for all 
$1\leq i\leq j\leq n$ (or $\{t_i\}_{i=1}^w$) and 
$z_i:=x_i (\prod_{k=1}^w t_k)^{r}$ for all $i=1,2,\cdots, n$. 
Since $B_{g,r}$ is a graded subalgebra of $A_g$, it is a connected 
graded domain. This is also a Clifford algebra over $T_g:=k[t_{(i,j)}]$ 
generated by $\{z_i\}_{i=1}^n$ subject to the relations
$$z_jz_i+z_iz_j= 2 (\prod_{k=1}^w t_k)^{2r} t_{(i,j)} =: 2 b_{ij}$$
from which the bilinear form $b$ and associated quadratic form $q$
can easily be recovered. In particular, $q(V)\subseteq 
(\prod_{k=1}^w t_k)^{2r} T_{g}$ where $V=\oplus_{i=1}^n T_{g} z_i$. 
By the definition of $A_g$, $\deg t_i=2$. 
Then $\deg z_i=1+4rw>3$. Now we assume that $N:=2r$ is bigger than 
$2r$ that is the degree of $d':=\det (t_{(i,j)})$. So we have 
$$n<r, \quad {\text{or equivalently}} \quad \deg d'<N$$ 
as required by \eqref{E8.9.1}. See also Remark \ref{xxrem8.10}.

Let $g$ be an algebra automorphism of $B_{g,d}$. 
By the above discussion, $g(t_i)$, for each $i$, is a linear 
combination of $\{t_j\}_{j=1}^w$ and $1$. Using the relations 
$z_i^2=b_{ii}$, we see that $\deg g(z_i)
=\deg (z_i)$ for all $i$. Thus $g$ must be a filtered automorphism
of $B_{g,d}$. 

Since $g$ preserves the discriminant $f$ and $f$ is homogeneous in 
$t_i$, $\deg g(t_i)=2$. Further, by using the expression of $f$ and 
the fact that $T_g$ is a UFD, 
$g(t_i)$ can not be a linear combination of $t_j$'s of more than
one term. Thus $g(t_i)=c_i t_j$ for some $j$ and some $c_i\in 
k^{\times}$. This implies that there is a permutation
$\sigma\in S_w$ and a collection of units $\{c_i\}_{i=1}^w$
such that $g(t_i)=c_i t_{\sigma(i)}$ for all $i$. Since $g$ is filtered
(by the last paragraph),
$g(z_i)=\sum_{k=1}^n d_{ik} z_k +e_i$ where $d_{ik}, e_i\in k$. 
Applying $g$ to the relation
$$z_i^2= b_{ii}=(\prod_{i=1}^w t_i)^N t_{\phi(i,i)}, \quad {\text{where}} \quad 
N:=2r,$$
we obtain that
$$(\sum_{k} d_{ik} z_k)^2+2 e_i (\sum_{k} d_{ik} z_k)+ e_i^2=
(\prod_{i=1}^w c_i t_i)^N g(t_{\phi(i,i)}).$$
Since $(\sum_{k} d_{ik} z_k)^2\in T$, we have 
$e_i (\sum_{k} d_{ik} z_k)=0$. Consequently, $e_i=0$ and 
$g(z_i)=\sum_{k=1}^n d_{ik} z_k$.
Applying $g$ to the relations
$$z_i z_j+z_jz_i=2b_{ij}=2(\prod_{i=1}^w t_i)^N t_{\phi(i,j)},$$
and expanding the left-hand side, we obtain that 
$$\sum_{k,l} d_{ik}d_{jl} (z_kz_l +z_lz_k)=
2(\prod_{i=1}^w c_i t_i)^N g(t_{\phi(i,j)}).$$
Hence $d_{ik}d_{jl}$ is nonzero for only one pair $(k,l)$.
Thus there is a set of units $\{d_i\}_{i=1}^n$ and a 
permutation $\psi\in S_{n}$ such that $g(z_i)=d_i z_{\psi(i)}$
for all $i=1,\cdots,n$. Then the above equation implies
that
$$
d_i d_j (\prod_{i=1}^w t_i)^N 
t_{\phi(\psi(i),\psi(j))}=(\prod_{i=1}^w c_i)^N (\prod_{i=1}^w t_i)^N
c_{\phi(i,j)} t_{\sigma (\phi(i,j))}$$
for all $i,j$. Therefore
\begin{equation}
\label{E8.9.2}\tag{E8.9.2}
\phi(\psi(i),\psi(j))=\sigma (\phi(i,j))
\end{equation}
and
\begin{equation}
\label{E8.9.3}\tag{E8.9.3}
d_i d_j =(\prod_{i=1}^w c_i)^N c_{\phi(i,j)}
\end{equation}
for all $i,j$.

By \eqref{E8.9.2}, $\sigma$ is completely determined by 
$\psi\in S_n$. Let $\bar{d}_i=d_i (\prod_{i=1}^w c_i)^{-r}$.
Then \eqref{E8.9.3} says that $\bar{d}_i \bar{d}_j=c_{\phi(i,j)}$.
So $\prod_{i=1}^w c_i=\prod_{1\leq i\leq j\leq n} \bar{d}_i \bar{d}_j$.
This means that $c_{\phi(i,j)}$s and $d_i$s are completely 
determined by $\bar{d}_i$s. In conclusion,
$$\Aut(B_{g,r})\cong\{\psi\in S_n\}\ltimes \{\bar{d}_i\in k^\times\mid i=1,\cdots,n\}
\cong S_n\ltimes (k^\times)^{n}.$$
In particular,  every algebra automorphism of $B_{g,r}$ is a graded algebra
automorphism.
\end{example}

\begin{remark}
\label{xxrem8.10}
As a consequence of the computation in Example \ref{xxex8.9},
$\Aut(B_{g,r})$ is independent of the parameter $r$ when
$r>n$. In fact, this assertion holds for all $r>0$, but its proof 
requires a different and longer analysis, so it is omitted. 
On the other hand, $\Aut(B_{g,0})=\Aut(A_g)$ is very different, 
see Remark \ref{xxrem3.9}(3). 
\end{remark}

We will work out one more automorphism group below.

\begin{example}
\label{xxex8.11} We continue to study Example \ref{xxex8.4} and
prove that every algebra automorphism of $A$ in Example \ref{xxex8.4} 
is graded. Some of unimportant details are omitted due to the 
length.

Claim 1: ${\mathfrak m}:=A_{\geq 1}$ is the only ideal of 
codimension 1 satisfying $\dim {\mathfrak m}/{\mathfrak m}^2=4$.
Suppose $I=(x_1-a_1,x_2-a_2,x_3-a_3,x_4-a_4)$ is an ideal of $A$ of
codimension 1 such that $\dim_k I/I^2=4$. 
Then the map $\pi: x_i\to a_i$ for all $i$
extends to an algebra homomorphism $A\to k$. Applying $\pi$ to the 
relations of $A$ in \eqref{E8.4.1}, we obtain that 
$$a_1 a_2=0, a_1 a_3=0, 2a_1a_4=a_3^2, a_2a_3=0, a_3a_4=0, a_2a_4=0.$$ 
Therefore $(a_i)$ is either $(a_1,0,0,0)$, or $(0,a_2,0,0)$, or
$(0,0,0,a_4)$. By symmetry, we consider the first case and the details
of the other cases are omitted. Let $z_i=x_i-a_i$ for all $i$. Then 
the first relation of \eqref{E8.4.1} becomes
$$z_1z_2+z_2z_1=(x_1-a_1)x_2+x_2(x_1-a_1)=-2a_1 x_2=-2a_1 z_2.$$
So $2a_1 z_2\in I^2$. Since $\dim I/I^2=4$, $a_1=0$. Thus we proved 
claim 1.

One of the consequences of claim 1 is that any algebra automorphism of 
$A$ preserves ${\mathfrak m}$. So we have a short exact sequence 
$$1\to \Aut_{uni}(A)\to \Aut(A)\to \Aut_{gr}(A)\to 1$$
where $\Aut_{gr}(A)$ is the group of graded algebra automorphisms 
of $A$ and $\Aut_{uni}(A)$ is the group of unipotent algebra 
automorphisms of $A$.

Claim 2: If $f$ is a nonzero normal element in degree 1, then 
$B:=A/(f)$ is Artin-Schelter regular domain of global dimension three.
By \cite[Lemma 1.1]{RZ}, $B$ has global dimension 3. Since $A$ satisfies
the $\chi$-condition \cite{AZ}, so is $B$. As 
a consequence, $B$ is AS regular of global dimension 3
\cite{AS}. It is well-known
that every Artin-Schelter regular algebra of global dimension three is
a domain (following by the Artin-Schelter-Tate-Van den Bergh's
classification \cite{AS,ATV1,ATV2}). 

Claim 3: If $f\in A_1$ is a normal element, then $f\in kx_2$ or 
$f\in kx_3$. First of all, both $x_2$ and $x_3$ are normal elements
by the relations \eqref{E8.4.1}. Note that $x_i g=\eta_{-1}(g) x_i$
for $i=2,3$, where $\eta_{-1}$ is the algebra automorphism of $A$ 
sending $x_i$ to $-x_i$ for all $i$. 

Suppose that $f$ is nonzero normal and $f\notin kx_3 \cup kx_4$. 
Then the image $\bar{f}$ of $f$ is normal in $A/(x_3)$. Since 
$A/(x_3)$ is a skew polynomial ring, by \cite[Lemma 3.5(d)]{KKZ}, 
$\bar{f}$ is a a scalar multiple of $x_i$ for some $i=1,2$, or $4$. 
This implies that $f$ is either $a x_1+b x_3$, or $ax_2+bx_3$ or 
$ax_4+bx_3$ for some $a,b\in k$. If $b=0$, then $f=x_1$ or $x_4$.
The relation $x_1x_4+x_4x_1=x_3^2$ implies that $A/(f)$ is not a
domain (as $x_3^2=0$ in $A/(f)$). This contradicts claim 2.
So the only possible case is $f=x_2$ (again yielding a contradiction). 
Now assume that $b\neq 0$ (and $a\neq 0$ because $f\notin kx_3 \cup 
kx_4$). We consider the first case and the details of the other
cases are similar and omitted. Since $f=ax_1+bx_3$, then 
relation $x_1x_3+x_3x_1=0$ implies that $x_1^2=0$ in $A/(f)$,
which contradicts with Claim 2. In all these cases, we obtain 
a contradiction, and therefore $f\in kx_2$ or $f\in kx_3$.

Since $A/(x_2)$ is not isomorphic to $A/(x_3)$, there is no
algebra automorphism sending $x_2$ to $x_3$. As a consequence, 
any graded automorphism $\psi$ of $A$ maps $x_2\to c_2 x_2$ and 
$x_3\to c_3 x_3$. Let $g$ be any graded algebra automorphism 
of $A$. Let $\bar{g}$ be the induced algebra automorphism 
of $A/(x_3)$. By \cite[Lemma 3.5(e)]{KKZ}, $\bar{g}$
sends $x_1\to c_1 x_1$ and $x_4\to c_4x_4$ or 
$x_1\to c_1 x_4$ and $x_4\to c_4x_1$. Then, by using the original
relations in \eqref{E8.4.1}, one can check that $g$ is of the form
$$x_1\to c_1 x_1, x_2\to c_2 x_2, x_3\to c_3 x_3, x_4\to c_4 x_4$$
where $c_1c_2=c_3^2=c_4^2$ or 
$$x_1\to c_1 x_4, x_2\to c_2 x_2, x_3\to c_3 x_3, x_4\to c_4 x_1$$
where $c_1c_2=c_3^2=c_4^2$. So 
$$\Aut_{gr}(A)\cong \{(c_1,c_2,c_3,c_4)\in (k^\times)^{4}\mid 
c_1c_2=c_3^2=c_4^2\}$$ 
which is completely determined.

Claim 4: $\Aut_{uni}(A)$ is trivial. 
Recall that the discriminant of $A$ over its center is 
$$d:=(x_2^2 x_3^2 (4x_1^2 x_4^2 -x_3^4))^8.$$
By Example \ref{xxex8.4}, the DDS subalgebra ${\mathbb D}(A)$ is the whole
algebra $A$.  The assertion follows 
from Theorem \ref{xxthm0.5}.

Combining all these claims, one sees that $\Aut(A)=\Aut_{gr}(A)$
which is described in Claim 3. 
\end{example}

\begin{remark}
\label{xxrem8.12}
Ideas as in Remark \ref{xxrem8.10} also apply to Example \ref{xxex6.3} 
and a similar conclusion holds. The interested reader can fill out the 
details.
\end{remark}

\subsection*{Acknowledgments} 
The authors would like to thank the referees for their careful 
reading and valuable comments.
A.A. Young was supported by the US
National Science Foundation (NSF Postdoctoral Research Fellowship,
No. DMS-1203744) and J.J. Zhang was supported by the US
National Science Foundation (Nos. DMS-0855743 and DMS-1402863).

\providecommand{\bysame}{\leavevmode\hbox to3em{\hrulefill}\thinspace}
\providecommand{\MR}{\relax\ifhmode\unskip\space\fi MR }
\providecommand{\MRhref}[2]{%

\href{http://www.ams.org/mathscinet-getitem?mr=#1}{#2} }
\providecommand{\href}[2]{#2}

\end{document}